\documentclass[12pt]{amsart} 
\usepackage{amssymb,amsmath,amsthm,color}
\usepackage{graphicx}
\usepackage{hyperref}
\usepackage{url} 
\usepackage{setspace}
\usepackage[margin=1in]{geometry}
\usepackage{comment}
\usepackage{mathrsfs}
\usepackage{amsthm}

\newtheorem{example}{Example}[section]
\newtheorem{theorem}{Theorem}[section]
\newtheorem{lemma}[theorem]{Lemma}
\newtheorem{proposition}[theorem]{Proposition}
\newtheorem{corollary}[theorem]{Corollary}
\newtheorem{remark}{Remark}[section]
\numberwithin{equation}{section}

\newcommand{\R}{\mathbb R}

\newcommand{\Z}{\mathbb Z} 
\newcommand{\N}{\mathbb N}

\title{Bimodal Wilson systems in $L^2(\R)$ }

\author{Divyang G. Bhimani}
\address{Department of Mathematics\\
University of Maryland\\
College Park\\
MD 20742}
\email{dbhimani@math.umd.edu}

\author{Kasso A. Okoudjou}
\address{Department of Mathematics and Norbert Wiener Center\\
University of Maryland\\
College Park\\
MD 20742}
\email{kasso@math.umd.edu}


\date{\today}


\begin{document}
\begin{abstract} Given  a window $\phi \in L^2(\R),$ and  lattice parameters $\alpha, \beta>0,$   we introduce a bimodal Wilson system $\mathcal{W}(\phi, \alpha, \beta)$ consisting of   linear combinations of at most two elements from an  associated Gabor $\mathcal{G}(\phi, \alpha, \beta)$.  For a class of window functions $\phi,$ we show that  the   Gabor system   $\mathcal{G}(\phi, \alpha, \beta)$ is  a tight  frame of   redundancy  $\beta^{-1}$ if and only if  the Wilson system $\mathcal{W}(\phi, \alpha, \beta)$ is Parseval system for  $L^2(\mathbb R).$ Examples of  smooth rapidly decaying generators $\phi$ are constructed. In addition, when $3\leq \beta^{-1}\in \N$, we prove that it is impossible to renormalize the elements of the constructed Parseval Wilson frame so as to get   a well-localized  orthonormal  basis for $L^2(\mathbb R)$. 
\end{abstract}
\maketitle \pagestyle{myheadings} \thispagestyle{plain}
\markboth{D. G. BHIMANI AND K. A. OKOUDJOU}{BIMODAL WILSON SYSTEMS}


\section{Introduction} \label{sec1}
Given that  $ \{e^{2\pi i m \cdot}: m\in \mathbb Z\}$
 forms an orthonormal 
basis (ONB) for $L^2([0, 1)),$ it is easy to establish that 
$$\mathcal{G}(\chi, 1,1)=\{\chi_{[0,1)}(\cdot-j)e^{2\pi i m \cdot}: j,m \in \mathbb Z\},$$
is an ONB for $L^2(\mathbb R)$, where 
 $\chi_{[0,1)}$ is the characteristic function of $[0,1).$ $\mathcal{G}(\chi, 1,1)$ is the simplest example of  Gabor systems,  first introduced in 1946 by D.~Gabor \cite{dg}. More generally, given $\alpha, \beta >0$ and $\phi\in L^2(\mathbb R)$, the set 
  \begin{equation}\label{gabsys}
 \mathcal{G}(\phi, \alpha, \beta)=\{ \phi_{j, m}(\cdot):=\phi(\cdot - \beta j)e^{2\pi i \alpha m \cdot}: \, j, m\in \mathbb Z\}
 \end{equation}
  is the Gabor system with generator (function) $\phi$ and (time-frequency) parameters $\alpha, \beta.$ 
 $ \mathcal{G}(\phi, \alpha, \beta)$ is called a Gabor frame if there exist $0<A\leq B$ such that for every $f\in L^2(\mathbb R)$ we have 
 \begin{equation}\label{frameineq}
 A\|f\|^2\leq \sum_{j, m \in \mathbb Z}|\langle f, \phi_{j, m}\rangle|^2\leq B\|f\|^2.
 \end{equation} 
A Gabor frame with $A=B$ is called a tight Gabor frame. In this case the frame bound $A$ will be referred to as the  \textbf{redundancy $A$}.  If in addition, $A=B=1$ we call the system a Parseval (Gabor) frame. We recall the following well-known result that will be used in the sequel, see \cite[Theorem 8.1]{CasChrJan01}, and \cite[Theorem 3.1]{wc}.

\begin{proposition}\label{wcza} 
Let $\phi \in L^2(\mathbb R)$ and $\alpha, \beta >0$. 
The  Gabor system  $\mathcal{G}(\phi, \alpha, \beta)$ is a  tight frame for $L^2(\mathbb R)$   with frame bound  $\beta^{-1}$  if and only if  $\phi$ satisfies 
$\sum_{ m \in \mathbb Z} \hat{\phi}(\xi-\alpha m) \overline{\hat{\phi}(\xi+\beta^{-1}k-\alpha m)} =\delta_{k,0}$ a.e. for each $k\in \mathbb Z.$
\end{proposition}

In addition,  the following result  about Parseval frames and ONBs will be used repeatedly, we refer to \cite[Section 7.1]{hw} for details.

\begin{proposition}\label{KS}
 Let  $ \{e_j\}_{j=1}^{\infty} \subset L^2(\mathbb R) $. The following statements hold. 
 \begin{enumerate}
 \item  \label{pl1}  For all $f \in L^2(\mathbb R),$ we have  \[\|f\|^2= \sum_{j=1}^{\infty} \left| \langle f, e_j \rangle \right|^2\] if and only if 
\[f= \sum_{j=1}^{\infty} \langle f, e_j \rangle e_j,\] with convergence in $L^2(\R),$ for all $f\in L^2(\R).$

\item \label{pl2}  If
\[ \|f\|^2= \sum_{j=1}^{\infty} \left| \langle f, e_j \rangle \right|^2 \]
holds for all $f$ in a dense subset $D\subset L^2(\mathbb R),$
then this equality holds for all $f\in L^2(\R).$
\item \label{hwt}
Suppose $\{ e_j: j=1,2,...\}$ is a Parseval frame. If  $\|e_j\|_{L^2}=1$ for all
$j$, then  $\{ e_j: j=1,2,...\}$ is an orthonormal basis for  $L^2(\R).$
\end{enumerate}    

  \end{proposition}

The characterization of the generators $\phi$ and the time-frequency parameters $\alpha, \beta$ such that $\mathcal{G}(\phi, \alpha, \beta)$ is a frame is still largely unresolved \cite{Gro14}. Nonetheless, it is known that if $\mathcal{G}(\phi, \alpha, \beta)$ is a Gabor  frame  then $0<\alpha \beta \leq 1.$  But when  $\alpha\beta>1$ the system in~\eqref{gabsys} is never complete. Furthermore,  $\mathcal{G}(\phi, \alpha, \beta)$ is an ONB for $L^2(\mathbb R)$ if and only if $\alpha \beta=1$.  For more details about these density results we refer to \cite[Section 7.5]{gro}, \cite{heil07}, and the references therein. 
It is also known that all Gabor ONB behave essentially like our first example in the sense that if $\mathcal{G}(\phi, \alpha, 1/\alpha)$ is an ONB, then, the window $\phi$ must be poorly localized in time or frequency that is $$\int_{\mathbb R}|x|^2|\phi(x)|^2\, dx=\infty \quad {\textrm or}\quad  \int_{\mathbb R}|\xi|^2 \, |\hat{\phi}(\xi)|^2\, d\xi=\infty$$ where $\hat{\phi}$ is the Fourier transform of $\phi$. This is the Balian-Low Theorem (BLT) that imposes strict limits on Gabor systems that form an ONB \cite{blt, Bat88, BenHeiWal94, Low85}.  

Introduced numerically by  K.~G.~Wilson \cite{ws}, the so-called  \emph{generalized Warnnier functions} have good time-frequency localization properties and thus are not subjected to the localization limits dictated  by the BLT. Latter, Daubechies, Jaffard, and Journ\'e formalized this definition and introduced what is now known as \emph{Wilson systems} \cite{djl}. Wilson ONBs have played major roles in some recent applications, including the detection of  the gravitational waves \cite{ChMJM16, NeKlMi12, Drag16}, or their use in electromagnetic  reflection-transmission problems in fiber optics \cite{floris2018wilson,floris2018electromagnetic} .

We now define the Wilson system for which each element $\psi_{j,m}$ is a linear combination of two Gabor functions localized at $(j, m)$ and $(j, -m)$ respectively. More precisely, given a Gabor system  $\mathcal{G}(\phi, \alpha, \beta)$, the associated (bimodal) Wilson system $\mathcal{W}(\phi, \alpha, \beta)$ is 

\begin{equation}\label{wilsys}
\mathcal{W}(\phi, \alpha, \beta)= \{ \psi_{j,m}: j \in \mathbb Z, m \in \mathbb N_0\}
\end{equation} where

\begin{align}\label{dks}
\psi_{j,m}(x)=\begin{cases}  \sqrt{2\beta}\phi_{2j, 0}(x)= \sqrt{2\beta} \phi(x-2\beta j) &  \text{if} \ \  j\in \mathbb Z, m=0,\\ 
\sqrt{\beta} \left[ e^{-2\pi i \beta j \alpha m}\phi_{j,m}(x) + (-1)^{j+m}e^{2\pi i \beta j \alpha m} \phi_{j,-m}(x)\right] & \text{if} \  \  (j,m)\in \mathbb Z \times \mathbb N.
\end{cases}
\end{align}
With these notations, the following result was proved in~\cite{djl}:

\begin{theorem}[\cite{djl}]\label{daub-wil} Let $\phi \in L^{2}(\mathbb R)$ be such that $\hat{\phi}(\xi) = \overline{\hat{\phi}(\xi)}$ and $\|\phi \|_2 = 1$. Then  the Gabor system $\mathcal{G}(\phi, 1, 1/2)$ 
is a tight frame for $L^{2}(\mathbb R)$ if, and only if, the Wilson system
$\mathcal{W}(\phi, 1, 1/2)$ 
is an orthonormal basis for $L^2(\mathbb R)$.
Furthermore, one can choose $\phi \in C^{\infty}(\mathbb R)$ with compact support. 
\end{theorem}

Theorem~\ref{daub-wil} has been generalized from the case of Gabor frames on the separable lattice $\Z \times \tfrac{1}{2}\mathbb Z $ to non separable lattices $A\mathbb Z^2$ where $A$ is any invertible matrix such that $|\textrm{det}A|=1/2$, see \cite{gt, pw}. The underlying theme in all these results is a one-to-one association of a tight Gabor frame of redundancy $(\alpha\beta)^{-1}=2$ with a bimodal  Wilson basis.  However,  it is still unknown whether similar associations can be made starting from a tight Gabor frame of other redundancy. For example, Gr\"ochenig  in \cite[p.168]{gro} posed the problem of the existence and construction  of a Wilson ONB starting from a tight Gabor frame with $\alpha=1$ and $\beta=1/3$. This problem is still unresolved. However, Wojdy\l\l o proved that taking linear combinations of three elements of a redundancy $3$ tight Gabor frame results in a (trimodal) Parseval Wilson frame \cite{pw1}. But the method developed was not constructive and it is not clear how to use it to produce an example of a well-localized window function $\phi$.  In higher dimensions, Wilson ONBs are usually constructed by taking tensor products of  $1$ dimensional Wilson ONBs. In this context,  (non-separable) Wilson ONBs for $L^2(\mathbb R^d)$ were recently constructed starting from tight Gabor frame of redundancy $2^k$ for each $k=0,1,2, \hdots, d$,\cite[Theorem 3.1 $\&$ Theorem 4.5]{bjlo}.

In this paper, we show that starting from a tight Gabor frame of redundancy $1/\beta$, one can construct a bimodal Parseval  Wilson frame. Furthermore, we can choose the generator to be a Schwartz function.  For example, as a consequence of some of our results we shall prove the following. 

\begin{theorem}\label{thm-main1} Let $\beta \in (0, 1/2)$. There exists $\phi \in S(\mathbb R)$ with $\hat{\phi} \in C_c^{\infty}(\mathbb R)$ such that 
the  Gabor system  $\mathcal{G}(\phi, 1, \beta)$ is a  tight frame for $L^2(\mathbb R)$   with frame bound  $\beta^{-1}$  if and only if  the Wilson system $\mathcal{W}(\phi, 1, \beta)$ is a Parseval  frame for $L^2(\mathbb R)$.
\end{theorem}

To convert this Wilson system into an ONB, one is left to normalize its elements to have unit $L^2$ norm. However, we prove that this is  impossible in general as the normalization conditions needed to get an ONB are incompatible with the definition of the Wilson system we use. In particular, our results suggest that for a redundancy $\beta^{-1}\in \N$ tight Gabor frame, the associated Wilson system should be  made of linear combinations of $\beta^{-1}$ elements from the Gabor frame.  It follows that  the bimodal Wilson system given by~\eqref{dks} where the coefficients in the linear combinations are the unimodular numbers $e^{-2\pi i \beta j \alpha m}$ and $(-1)^{j+m}e^{2\pi i \beta j \alpha m}$ can never lead to an ONB. 

\begin{theorem}\label{thm-main2} Let $3\leq \beta^{-1}\in \N$. There exists no function $\phi \in L^2(\mathbb R)$ with either $\hat{\phi}$ compactly supported, or $\phi$ and $\hat{\phi}$ having exponential decay, 
such that the Wilson system $\mathcal{W}(\phi, 1, \beta)$ is an ONB  for $L^2(\mathbb R)$.
\end{theorem}

We recall that the space of smooth functions on $\mathbb R$ with compact support is denoted by $C_{c}^{\infty}(\mathbb R)$,  the Schwartz class is $\mathcal{S}(\mathbb R)$, the space of tempered distributions is $\mathcal{S'}(\mathbb R).$ 
The (unitary) $L^2$ Fourier transform  is defined by 
\begin{eqnarray*}
\mathcal{F}f(w)=\widehat{f}(w)=\int_{\mathbb R} f(t) e^{- 2\pi i t w}dt, \  w\in \mathbb R, 
\end{eqnarray*}
with inverse given by \begin{eqnarray*}
\mathcal{F}^{-1}f(x)=f^{\vee}(x)=\int_{\mathbb R} f(w)\, e^{2\pi i x w} dw,~~x\in \mathbb R.
\end{eqnarray*}
The torus $\{ z\in \mathbb C: |z|=1\}$ is denoted by $\mathbb T.$ If $f\in L^1(\mathbb T)$, we define it's Fourier coefficients by
 $$\widehat{f}(m)=\int_{\mathbb T}f(x)e^{-2\pi i mx} dx,   \ (m\in \mathbb Z).$$

The rest of the paper is organized as follows. Section~\ref{b3} contains the technical results needed to prove our main results. In particular,  we derive necessary  and sufficient conditions on $\phi$ for the $\{\psi_{j,m}\}$   to be an ONB for $L^2(\mathbb R).$ In Section~\ref{cgfwb} we state and prove one of our main results Theorem~\ref{wg}. In particular, we give necessary and sufficient conditions to turn a tight Gabor frame into a Parseval Wilson system. We also  indicate under which extra condition this Wilson system becomes an ONB, and provide examples of generators $\phi \in \mathcal{S}(\R)$. Finally, in Section~\ref{ztrc} we use the Zak transform to construct more examples of generator $\phi \in  \mathcal{S}(\R)$ such that $\phi$ and $\hat{\phi}$ have exponential decay.

\section{Characterization for  Wilson bases in $L^2(\mathbb R)$}\label{b3}
In this section we find necessary and sufficient conditions on $\phi$ that guarantee that  the Wilson system $\mathcal{W}(\phi, \alpha, \beta)$ forms a Parseval frame, Theorem~\ref{chwb}. 
In addition, by  normalizing each vector in $\mathcal{W}(\phi, \alpha, \beta)$ we find additional conditions needed to make this Parseval (Wilson) frame an ONB.

\begin{theorem}\label{chwb} Let $\alpha, \beta >0,$ and $\{\psi_{j,m}\}_{j\in \mathbb Z,  m\in \mathbb N_0}$ is defined by \eqref{dks}. The following statements are equivalent:
\begin{enumerate}
\item[(a)] $\mathcal{W}(\phi, \alpha, \beta)=\{\psi_{j,m}\}_{j\in \mathbb Z,  m\in \mathbb N_0}$ is a Parseval frame for $L^2(\R)$. 
\item[(b)] $\Phi_k(\xi)=\delta_{k, 0}\, a.e.$,  and $\Delta_k(\xi)=0\, a.e.$ for each $k\in \Z$, where 
\begin{equation*}
\begin{cases}
\Phi_k(\xi)= \sum_{ m \in \mathbb Z} \hat{\phi}(\xi-\alpha m) \overline{\hat{\phi}(\xi+\beta^{-1}k-\alpha m)},\\
\Delta_k(\xi)=  \sum_{ m \in \mathbb Z} (-1)^m\hat{\phi}(\xi+\alpha m)\overline{\hat{\phi}(\xi+\beta^{-1}(k +1/2)-\alpha m)}.
\end{cases}
\end{equation*}
\end{enumerate}
\end{theorem}

As an immediate consequence of this result we have.

\begin{corollary}\label{cor-par-onb}
 Let $\alpha, \beta >0,$ and $\{\psi_{j,m}\}_{j\in \mathbb Z,  m\in \mathbb N_0}$ is defined by \eqref{dks}.  Suppose that one of the statements (a) or (b) in Theorem~\ref{chwb} hold  (hence all of them hold), then $\{\psi_{j,m}\}_{j\in \mathbb Z,  m\in \mathbb N_0}$ is  an ONB for $L^2(\R)$ if and only if 
 \begin{equation*}
 \begin{cases}
 \|\phi\|_{L^2}= \frac{1}{\sqrt{2\beta}},\\
  \Re    \langle X_{j,m}, Y_{j,m}\rangle=0.
  \end{cases}
  \end{equation*}
\end{corollary}

\noindent
In order to  prove Theorem~\ref{chwb}, and for the future reference,  first we  note that  the Fourier transform  $ \widehat{\phi_{j,m}}$ of $\phi_{j,m}$  is 

\begin{eqnarray}\label{fabtf}
\widehat{\phi_{j,m}}(\xi)= e^{-2\pi i \beta j (\xi -\alpha m)}\hat{\phi}(\xi-\alpha m), \ (\xi \in \mathbb R),
\end{eqnarray} 
and the  Fourier transform  $\widehat{\psi_{j,m}}$ of $\psi_{j,m}$ is
\begin{align}\label{fdks}
\widehat{\psi_{j,m}} (\xi) =\begin{cases} \sqrt{2\beta} e^{-4\pi i \beta j  \xi}\hat{\phi}(\xi)
& \text{if} \  \  j\in \mathbb Z, m=0,\\ 
\sqrt{\beta}\left[e^{-2\pi i \beta j \xi }\hat{\phi}(\xi-\alpha m) + (-1)^{j+m}e^{-2\pi i \beta j \xi}\hat{\phi}(\xi+\alpha m)\right] & \text{if} \  \  (j,m)\in \mathbb Z \times \mathbb N.
\end{cases}
\end{align}

\begin{remark} In  \cite{kw2}, using the notations $\phi_{j,m}(x)= e^{ixma}\phi(x-bj)$ where  $a, b>0$, the following Wilson-type system was considered.
\begin{equation}\label{wso}
\psi_{j, m}(x) = \begin{cases}  \phi(x-2bj) &  \text{if} \ j \in \mathbb Z, m=0,\\
2^{-1/2}[\phi_{j,m}(x)+ (-1)^{j+m} \phi_{j, -m}(x)] & \text{if} \ j \in \mathbb Z, m \in \mathbb N.
\end{cases}
\end{equation}
In particular, when $a=\pi$ and $b=1$   \cite[Theorem 1.2]{kw2} which is similar to Theorem~\ref{chwb}   was proved, and it was claimed that the proof extends to all  $a, b>0.$ However, this is not the case because of the choice of coefficients in defining the Wilson-type system~\eqref{wso}. Indeed, in the Fourier domain  (using the normalization $\hat{\psi}(\xi)=\int \psi(x) e^{-i x\cdot \xi} d\xi$), ~\eqref{wso}  becomes 
\begin{equation}\label{wsf}
\widehat{\psi_{j, m}}(\xi) = \begin{cases}  e^{-2ibj \xi}\hat{\phi}(\xi) &  \text{if} \ j \in \mathbb Z, m=0,\\
2^{-1/2}[e^{-ibj (\xi-ma)} \hat{\phi}(\xi -ma)+ (-1)^{j+m} e^{-ibj (\xi+ma)} \hat{\phi}(\xi +ma)] & \text{if} \ j \in \mathbb Z, m \in \mathbb N.
\end{cases}
\end{equation}
When $a=\pi$ and $b=1$  the term $e^{\pm ibj ma} =\pm1$, which is why  \cite[Theorem 1.2]{kw2} holds. However, when $ab\neq \pi$ this is no longer the case and the proposed system cannot be an ONB. 
We resolve this problem by introducing in our proposed Wilson system~\eqref{dks}  where the unimodular term $e^{\pm 2\pi i \beta j \alpha m}$  allows for the cancellations needed to establish our results.  
\end{remark}

The proof of Theorem~\ref{chwb}  will follow from  Lemma  \ref{tl1}, and Proposition \ref{dp}, which we first state and prove. 

First, observe that by  Proposition~\ref{KS}\eqref{pl1},   $\mathcal{W}(\phi, \alpha, \beta)=\{\psi_{j,m}\}_{j\in \mathbb Z,  m\in \mathbb N_0}$ is a Parseval frame for $L^2(\R)$ if and only  for each $f\in L^2(\R),$ we have 
\begin{equation}\label{eq-parse}
\|f\|^2_{L^2}=\sum_{m\in \mathbb N_0} \sum_{j\in \mathbb Z} |\langle f, \psi_{j,m}  \rangle|^2.
\end{equation}
 So by Proposition~\ref{KS}\eqref{pl2},  to establish Theorem~\ref{chwb} it is enough to prove that part (b) is equivalent to~\eqref{eq-parse} for all $f$ belonging to a dense  subset, $\mathcal{D}$ of $L^2(\mathbb R).$  Here and in the sequel, we choose 
\[\mathcal{D}=\left\{ f\in L^2(\mathbb R): \hat{f}\in L^{\infty}(\mathbb R) \ \text{and  support of} \ \hat{f} \ \text{  is a compact subset of} \ \mathbb{R}\setminus \{0\} \right\}.\]

In the next proposition we set $$\mathcal{I}(f)= \sum_{m\in \mathbb N_0} \sum_{j\in \mathbb Z} |\langle f, \psi_{j,m} \rangle|^2,$$
$$\mathcal{I}_0(f)= \int_{\mathbb R} \sum_{k\in \mathbb Z} \hat{f}(\xi+\beta^{-1}k)\overline{\hat{f}(\xi)} \Phi_k(\xi) d\xi,$$
and
$$ \mathcal{I}_1(f)= \int_{\mathbb R} \sum_{k\in \mathbb Z}\overline{\hat{f}(\xi)} \hat{f}(\xi+\beta^{-1}( k +1/2)) \Delta_k(\xi) d\xi.$$
With these notations we have. 
\begin{proposition} \label{dp} Let $\alpha, \beta>0$ and $\phi\in L^2(\mathbb R)$. For any $f\in \mathcal{D}$  we have the  following decomposition
  \begin{eqnarray*}
\mathcal{I}(f)=\mathcal{I}_0(f)+ \mathcal{I}_1(f).
 \end{eqnarray*}
\end{proposition}
\begin{proof}
Plancherel's Theorem  together with~\eqref{fdks} give  
\begin{eqnarray*}\label{nd1}
 \mathcal{I}(f) & = &  \sum_{m\in \mathbb N_0} \sum_{j\in \mathbb Z} |\langle f, \psi_{j,m} \rangle|^2=  \sum_{m\in \mathbb N_0} \sum_{j\in \mathbb Z} |\langle \hat{f}, \widehat{\psi_{j,m}} \rangle|^2\nonumber \\
& = &  \sum_{j\in \mathbb Z} |\langle \hat{f}, \widehat{\psi_{j,0}} \rangle|^2+\sum_{ m\in \mathbb N} \sum_{j\in \mathbb Z} |\langle \hat{f}, \widehat{\psi_{j,m}} \rangle|^2\nonumber\\
& = &2\beta \sum_{j\in \mathbb Z}  \left| \int_{\mathbb R} \hat{f}(\xi) \overline{\hat{\phi}(\xi)}e^{2\pi i (2\beta j)  \xi} d\xi \right|^2\nonumber \\
&&+ \beta\sum_{ m\in \mathbb N} \sum_{j\in \mathbb Z} \left| \int_{\mathbb R} \hat{f}(\xi) \left[e^{2\pi i \beta j  \xi }\overline{\hat{\phi}(\xi-\alpha m)}
 + (-1)^{j+m}e^{2\pi i \beta j  \xi}\overline{\hat{\phi}(\xi+\alpha m)} \ \right]d\xi \right|^2. 
\end{eqnarray*}
For fix $ m\in \mathbb N,$ set 
\begin{eqnarray}\label{sn}
F_{\alpha m, \beta}(\xi)=\hat{f}(\beta^{-1}\xi)\overline{\hat{\phi}(\beta^{-1}\xi-\alpha m)} \ \  \text{and} \  \ 
F_{0, \beta}(\xi)=\hat{f}((2\beta)^{-1}\xi)\overline{\hat{\phi}((2\beta)^{-1}\xi)}.
\end{eqnarray}

Since  $f\in \mathcal{D}, F_{\alpha m, \beta}$ is compactly supported in  $\mathbb R\setminus \{0\}$ and belongs to $L^1(\mathbb R)\cap L^2(\mathbb R).$ 
By a  simple change of variables and in view of \eqref{sn}, we may rewrite

\begin{align}\label{c1}
\begin{cases}
\int_{\mathbb R} \hat{f}(\xi)\,  \overline{\hat{\phi}(\xi-\alpha m)}\,  e^{2\pi i \beta j  \xi } \, d\xi  &=
  \beta^{-1}\widehat{F_{\alpha m, \beta}}(-j),\\
  \int_{\mathbb R} \hat{f}(\xi)\,  \overline{\hat{\phi}(\xi+\alpha m)} \, e^{2\pi i \beta j  \xi } \, d\xi &= \beta^{-1} \widehat{F_{-\alpha m, \beta}}(-j) ,\\
  \int_{\mathbb R} \hat{f}(\xi) \, \overline{\hat{\phi}(\xi)} \, e^{2\pi i (2\beta)j  \xi } \, d\xi  
&=   (2\beta)^{-1}\widehat{F_{0, \beta}}(-j).
  \end{cases}
 \end{align}

In view of~\eqref{c1}, we can write 
\begin{eqnarray}\label{nd2}
\mathcal{I}(f)& = & \frac{1}{2\beta} \sum_{j\in \mathbb Z}  \left| \widehat{F_{0, \beta}}(-j) \right|^2 \nonumber \\
&&+\frac{1}{\beta}\sum_{ m\in \mathbb N} \sum_{j\in \mathbb Z} \left| \widehat{F_{\alpha m, \beta}}(-j)+ (-1)^{j+m} \widehat{F_{-\alpha m, \beta}}(-j) \right|^2\nonumber\\
& = & \frac{1}{2\beta} \sum_{j\in \mathbb Z} \left| \widehat{F_{0, \beta}}(-j) \right|^2 +\frac{1}{\beta}  \sum_{ m\in \mathbb N} \sum_{j\in \mathbb Z}  I_{m,j}\nonumber\\
& = &  I_0 + I_1 +I_2+I_3+I_4,
\end{eqnarray}
where 
\begin{eqnarray*}
 I_{m,j} & = & \left| \widehat{F_{\alpha m, \beta}}(-j)+ (-1)^{j+m}\widehat{F_{-\alpha m, \beta}}(-j) \right|^2 \\
& =  &  \widehat{F_{\alpha m, \beta}}(-j)\overline{ \widehat{F_{\alpha m, \beta}}(-j)}+ (-1)^{j+m}\widehat{F_{\alpha m, \beta}}(-j)\overline{ \widehat{F_{-\alpha m, \beta}}(-j)}\\
&& + (-1)^{j+m}\widehat{F_{-\alpha m, \beta}}(-j)\overline{ \widehat{F_{\alpha m, \beta}}(-j)} +\widehat{F_{-\alpha m, \beta}}(-j)\overline{ \widehat{F_{-\alpha m, \beta}}(-j)},
\end{eqnarray*}
and $I_{i}$ ($i=0,1,2,3,4$) will be  introduced latter. 

 Using the fact that $\mathbb R= \cup_{k\in \mathbb Z} (k+Q)=\cup_{k\in \mathbb Z} (k+ [-\frac{1}{2}, \frac{1}{2}))$, and that $F_{\alpha m, \beta}$ is compactly supported,  we obtain
\begin{eqnarray}\label{fcc}
\widehat{F_{\alpha m, \beta}}(-j) & = & \int_{ \mathbb R} F_{\alpha m, \beta}(\xi) e^{2\pi i j\xi} d\xi=\sum_{k\in \mathbb Z}\int_{Q -k}F_{\alpha m, \beta}(\xi) e^{2\pi i j \xi} d\xi \nonumber \\
& = & \int_{Q}  \left( \sum_{k\in \mathbb Z}F_{\alpha m, \beta}(\xi+k) \right) e^{2\pi i j (\xi+k)} d\xi\nonumber\\
& = & \int_{\mathbb T}  H(\xi) e^{2\pi i j \xi} d\xi,
\end{eqnarray}
where  $H(\xi)=\sum_{k\in \mathbb Z}F_{\alpha m, \beta}(\xi+k).$
We note  that $H(\xi)$ is a $1-$periodic function, and since $F_{\alpha m, \beta}$ is compactly supported, it follows   that $H\in L^2(\mathbb T)$, and it's Fourier coefficients  are $\widehat{F_{\alpha m, \beta}}(-j)$.
By the  Parseval's theorem, we have
\begin{eqnarray}\label{h1}
\sum_{j\in \mathbb Z} \left| \widehat{F_{\alpha m, \beta}}(-j)\right|^2 & =  & \int_{\mathbb T}\left( \sum_{k\in \mathbb Z} F_{\alpha m, \beta}(\xi +k ) \right)\overline{ \left( \sum_{p\in \mathbb Z} F_{\alpha m, \beta}(\xi+ p ) \right)} d\xi \nonumber \\
& = & \int_{\mathbb R}\left( \sum_{k\in \mathbb Z} F_{\alpha m, \beta}(\xi +k ) \right)\overline{ F_{\alpha m, \beta}(\xi) } d\xi.
\end{eqnarray} In view of \eqref{fcc}, and  by  the Poisson summation formula, we obtain
\begin{eqnarray}\label{h2}
D & := & \sum_{j\in \mathbb Z}(-1)^{j}\widehat{F_{\alpha m, \beta}}(-j)\overline{ \widehat{F_{-\alpha m, \beta}}(-j)} \nonumber\\
& = &  \int_{\mathbb T} \overline{ \left( \sum_{k\in \mathbb Z}F_{-\alpha m, \beta}(\xi+k) \right)} \cdot  \left(\sum_{j\in \mathbb Z}\widehat{F_{\alpha m, \beta}}(-j)e^{-2\pi i j  (1/2+ \xi)} \right) d\xi\nonumber\\
& = & \int_{\mathbb T} \overline{ \left( \sum_{k\in \mathbb Z}F_{-\alpha m, \beta}(\xi+k) \right)}  \cdot  \left(\sum_{j\in \mathbb Z} F_{\alpha m, \beta}(\xi + j  +1/2) \right) d\xi\nonumber\\
& = & \int_{\mathbb R} \overline{ F_{-\alpha m, \beta}(\xi)}  \cdot \left( \sum_{j\in \mathbb Z} F_{\alpha m, \beta}(\xi + j  +1/2) \right) d\xi.
\end{eqnarray}
By a simple  change of variable, \eqref{h1}, and  \eqref{sn},  we have 
\begin{eqnarray*}
I_0 & = & \frac{1}{2\beta} \sum_{j\in \mathbb Z} \left| \widehat{F_{0, \beta}}(-j)\right|^2 =\frac{1}{2\beta} \int_{\mathbb R}\left( \sum_{k\in \mathbb Z} F_{0, \beta}(\xi +k ) \right)\overline{ F_{0, \beta}(\xi) } d\xi\\
& = &  \int_{\mathbb R}\left( \sum_{k\in \mathbb Z} \hat{f}(\xi+ (2\beta)^{-1}k))\overline{\hat{\phi}(\xi + (2\beta)^{-1}k))}\right) \overline{\hat{f}(\xi)}\hat{\phi}(\xi) d\xi\\
& = &  \int_{\mathbb R}\left( \sum_{k\in \mathbb Z} \hat{f}(\xi+ \beta^{-1}k)\overline{\hat{\phi}(\xi + \beta^{-1}k))}\right) \overline{\hat{f}(\xi)}\hat{\phi}(\xi) d\xi\\
&& +  \int_{\mathbb R}\left( \sum_{k\in \mathbb Z} \hat{f}(\xi+ \beta^{-1}(k+1/2))\overline{\hat{\phi}(\xi + \beta^{-1}(k+1/2))}\right) \overline{\hat{f}(\xi)}\hat{\phi}(\xi) d\xi\\
& = & I_0'+I_0''.
\end{eqnarray*}
Similarly, in view of \eqref{h2},  we have 
\begin{eqnarray*}
I_1 & = & \frac{1}{\beta} \sum_{ m \in \mathbb N}\sum_{j\in \mathbb Z} \left| \widehat{F_{\alpha m, \beta}}(-j)\right|^2 \\
& = & \frac{1}{\beta}\sum_{ m \in \mathbb N} \int_{\mathbb R}\left( \sum_{k\in \mathbb Z} F_{\alpha m, \beta}(\xi +k ) \right)\overline{ F_{\alpha m, \beta}(\xi) } d\xi.\\
& = & \frac{1}{\beta} \sum_{ m \in \mathbb N} \int_{\mathbb R}\left( \sum_{k\in \mathbb Z} \hat{f}(\beta^{-1}(\xi+k))\overline{\hat{\phi}(\beta^{-1}(\xi+k)-\alpha m)} \right)\overline{\hat{f}(\beta^{-1}\xi)}\hat{\phi}(\beta^{-1}\xi-\alpha m) d\xi\\
& = & \sum_{ m \in \mathbb N} \int_{\mathbb R}\left( \sum_{k\in \mathbb Z} \hat{f}(\xi+ \beta^{-1}k)\overline{\hat{\phi}(\xi+\beta^{-1}k-\alpha m)} \right)\overline{\hat{f}(\xi)}\hat{\phi}(\xi-\alpha m) d\xi, 
\end{eqnarray*}
and 
\begin{eqnarray*}
I_2 & = &
\frac{1}{\beta}\sum_{ m\in \mathbb N}\sum_{j\in \mathbb Z}(-1)^{j+m}\widehat{F_{\alpha m, \beta}}(-j)\overline{ \widehat{F_{-\alpha m, \beta}}(-j)}\\
 & = &  \frac{1}{\beta}\sum_{ m \in \mathbb N}(-1)^m \int_{\mathbb R} \overline{ F_{-\alpha m, \beta}(\xi)}  \cdot \left( \sum_{j\in \mathbb Z} F_{\alpha m, \beta}(\xi + j  +1/2) \right) d\xi\\
 & = & \frac{1}{\beta} \sum_{ m \in \mathbb N}(-1)^m \int_{\mathbb R}\overline{\hat{f}(\beta^{-1}\xi)}\hat{\phi}(\beta^{-1}\xi+\alpha m) \\
 && \cdot \left( \sum_{j\in \mathbb Z}\hat{f}(\beta^{-1}(\xi + j  +1/2))\overline{\hat{\phi}(\beta^{-1}(\xi + j +1/2)-\alpha m)}
 \right) d\xi\\
 & = &  \sum_{ m \in \mathbb N}(-1)^m \int_{\mathbb R}\overline{\hat{f}(\xi)}\hat{\phi}(\xi+\alpha m) \\
 && \cdot \left( \sum_{j\in \mathbb Z}\hat{f}( \xi +\beta^{-1}( j  +1/2))\overline{\hat{\phi}(\xi+\beta^{-1}( j +1/2)-\alpha m)}
 \right) d\xi.\\
\end{eqnarray*}
By similar arguments, we have 
\begin{eqnarray*}
I_3 & = &
\frac{1}{\beta}\sum_{ m\in \mathbb N}\sum_{j\in \mathbb Z}(-1)^{j+m}\widehat{F_{-\alpha m, \beta}}(-j)\overline{ \widehat{F_{\alpha m, \beta}}(-j)}\\
 & = & \frac{1}{\beta} \sum_{ m \in \mathbb N}(-1)^m \int_{\mathbb R}\overline{\hat{f}(\beta^{-1}\xi)}\hat{\phi}(\beta^{-1}\xi-\alpha m)) \\
 && \cdot \left( \sum_{j\in \mathbb Z}\hat{f}(\beta^{-1}(\xi + j +1/2))\overline{\hat{\phi}(\beta^{-1}(\xi + j   +1/2)+\alpha m)}
 \right) d\xi\\
  & = &  \sum_{m \in \mathbb N}(-1)^m \int_{\mathbb R}\overline{\hat{f}(\xi)}\hat{\phi}(\xi-\alpha m)) \\
 && \cdot \left( \sum_{j\in \mathbb Z}\hat{f}(\xi+ \beta^{-1}( j +1/2))\overline{\hat{\phi}(\xi +\beta^{-1}( j   +1/2)+\alpha m)}
 \right) d\xi,
\end{eqnarray*}
and 
\begin{eqnarray*}
I_4 & = & \frac{1}{\beta}\sum_{ m \in \mathbb N}\sum_{j\in \mathbb Z} \left| \widehat{F_{-\alpha m, \beta}}(-j)\right|^2 \\
& = & \frac{1}{\beta} \sum_{m \in \mathbb N} \int_{\mathbb R}\left( \sum_{k\in \mathbb Z} \hat{f}(\beta^{-1}(\xi+k)\overline{\hat{\phi}(\beta^{-1}(\xi+k)+\alpha m)} \right)\overline{\hat{f}(\beta^{-1}\xi)}\hat{\phi}(\beta^{-1}\xi+\alpha m) d\xi\\
& = &  \sum_{ m \in \mathbb N} \int_{\mathbb R}\left( \sum_{k\in \mathbb Z} \hat{f}(\xi+\beta^{-1}k)\overline{\hat{\phi}(\xi+\beta^{-1}k+\alpha m)} \right)\overline{\hat{f}(\xi)}\hat{\phi}(\xi+\alpha m) d\xi.
\end{eqnarray*}

We shall justify in~Lemma \ref{tl1} below the  change of  the orders of integration and summation in next few steps. Consequently,

\begin{eqnarray*}\label{i3}
I_0'+I_1+ I_4 & = &  \sum_{ m \in \mathbb Z} \int_{\mathbb R}\left( \sum_{k\in \mathbb Z} \hat{f}(\xi+\beta^{-1}k)\overline{\hat{\phi}(\xi+ \beta^{-1}k-\alpha m)} \right)\overline{\hat{f}(\xi)}\hat{\phi}(\xi-\alpha m) d\xi\nonumber\\
& = & \int_{\mathbb R}\left( \sum_{k\in \mathbb Z} \hat{f}(\xi+ \beta^{-1}k)\overline{\hat{f}(\xi)}\right) \sum_{ m \in \mathbb Z} \hat{\phi}(\xi-\alpha m) \overline{\hat{\phi}(\xi+\beta^{-1}k-\alpha m)} d\xi\nonumber\\
& = &  \int_{\mathbb R} \sum_{k\in \mathbb Z} \hat{f}(\xi+\beta^{-1}k)\overline{\hat{f}(\xi)} \Phi_k(\xi) d\xi,
\end{eqnarray*}
and
\begin{eqnarray*}\label{i2}
I_0''+I_2+I_3 & = &  \sum_{ m \in \mathbb Z}(-1)^m \int_{\mathbb R}\overline{\hat{f}(\xi)}\hat{\phi}(\xi+\alpha m)\nonumber \\
 && \cdot \left( \sum_{j\in \mathbb Z}\hat{f}(\xi +\beta^{-1}(\ j  +1/2))\overline{\hat{\phi}(\xi +\beta^{-1}( j +1/2)-\alpha m)}
 \right) d\xi\nonumber\\
 & = & \int_{\mathbb R} \sum_{k\in \mathbb Z}\overline{\hat{f}(\xi)} \hat{f}(\xi +\beta^{-1}(k  +1/2)) \Delta_k(\xi) d\xi.
\end{eqnarray*}
This together  with  \eqref{nd2}, we obtain
\begin{eqnarray*}
\mathcal{I}(f) & = &  \int_{\mathbb R} \sum_{k\in \mathbb Z} \hat{f}(\xi+\beta^{-1}k)\overline{\hat{f}(\xi)} \Phi_k(\xi) d\xi\\
&& + \int_{\mathbb R} \sum_{k\in \mathbb Z}\overline{\hat{f}(\xi)} \hat{f}(\xi +\beta^{-1}( k  +1/2)) \Delta_k(\xi) d\xi.
\end{eqnarray*}
This completes the proof.
\end{proof}
The following technical result justifies the change of the order of integration and summation performed in the proof of   Proposition~\ref{dp}.

\begin{lemma}\label{tl1} Let $\alpha, \beta >0.$  If $f\in \mathcal{D}$ and $\phi \in L^2(\mathbb R),$  then
\begin{eqnarray}\label{tle1}
 \sum_{ m \in \mathbb Z} \int_{\mathbb R}\left( \sum_{k\in \mathbb Z} |\hat{f}(\xi+\beta^{-1}k)| \left| \hat{\phi}(\xi+ \beta^{-1}k-\alpha m) \right|\right) | \hat{f}(\xi)| |\hat{\phi}(\xi-\alpha m)| d\xi < \infty,
\end{eqnarray}
and \begin{eqnarray}\label{tle2}
K & : = &\sum_{ m \in \mathbb Z} \int_{\mathbb R}| \hat{f}(\xi)| |\hat{\phi}(\xi+\alpha m)| \\
&& \cdot  \left( \sum_{j\in \mathbb Z}|\hat{f}(\xi +\beta^{-1}(\ j  +1/2))| \left| \hat{\phi}(\xi +\beta^{-1}( j +1/2)-\alpha m)
 \right|\right) d\xi < \infty.\nonumber
\end{eqnarray}
\end{lemma}

To prove Lemma \ref{tl1} \eqref{tle1}, it suffices to show that 
\begin{eqnarray}\label{stl2}
\int_{\mathbb R} \sum_{k\in \mathbb Z} \sum_{ m \in \mathbb Z}   |\hat{f}(\xi+ \beta^{-1}k)| \left| \hat{f}(\xi) \right| |\hat{\phi}(\xi-\alpha m) |^2 d\xi< \infty.
\end{eqnarray}
This is because 
\[ \left |2 \hat{\phi}(\xi-\alpha m) \hat{\phi}(\xi+\beta^{-1}k-\alpha m)\right|\leq |\hat{\phi}(\xi-\alpha m) |^2+ \left|\hat{\phi}(\xi+\beta^{-1}k-\alpha m)\right|^2.\]

We remark that the summation involving $\left|\hat{\phi}(\xi+\beta^{-1}k-\alpha m)\right|^2$ reduced to \eqref{stl2} via the change of variable $\xi\mapsto \xi -\beta^{-1}k .$ And \eqref{stl2} is an immediate consequence of the following lemma

\begin{lemma}\label{tl2} Suppose $0<a< b < \infty, \hat{f}\in L^{\infty}(\mathbb R), $ and $   \text{supp} \ \hat{f} \subset \{ \xi : a< |\xi|< b \},$  then 
\[ \Sigma(\xi)= \sum_{k, m \in \mathbb Z}  |\hat{f}(\xi+ \beta^{-1}k+\alpha m)| | \hat{f}(\xi+\alpha m) | \lesssim \|\hat{f}\|_{L^{\infty}},\]
for almost every $\xi \in \mathbb R.$ 
 \end{lemma}

 \begin{proof}
When $|\beta^{-1}k|>\delta=b-a,$ we have $\left| \hat{f}(\xi+ \beta^{-1}k+\alpha m) \hat{f}(\xi+\alpha m) \right|=0.$ Then
\begin{eqnarray*}
\Sigma(\xi) &  \leq  & \sum_{m \in \mathbb Z} \sum_{|\beta^{-1}k|\leq  \delta} \left| \hat{f}(\xi+ \beta^{-1}k+\alpha m) \hat{f}(\xi+\alpha m) \right|\\
& \lesssim & C_{\delta} \sum_{m\in \mathbb Z}\left| \hat{f}(\xi+\alpha m)\right| \|\hat{f}\|_{L^{\infty}} \lesssim \|\hat{f}\|_{L^{\infty}}.
\end{eqnarray*}
\end{proof}
Since
\[ \left |2 \hat{\phi}(\xi+\alpha m) \hat{\phi}(\xi+\beta^{-1}(k+1/2)-\alpha m)\right|\leq |\hat{\phi}(\xi+\alpha m) |^2+ \left| \hat{\phi}(\xi+\beta^{-1}(k+1/2)-\alpha m)\right|^2,\]
we note,  to prove Lemma \ref{tl1} \eqref{tle2}, it suffices to prove
\begin{eqnarray}\label{stl3}
\int_{\mathbb R} \sum_{k\in \mathbb Z} \sum_{ m \in \mathbb Z} |\hat{f}(\xi+ \beta^{-1}(k+1/2))| |\hat{f}(\xi)| |\hat{\phi}(\xi+\alpha m) |^2 d\xi< \infty.
\end{eqnarray}

It is clear that  the summation involving $\left|\hat{\phi}(\xi+\beta^{-1}(k+1/2)-\alpha m)\right|^2$ reduces to \eqref{stl3} via the change of variable $\xi\mapsto \xi -\beta^{-1}(k+1/2) .$ And \eqref{stl3} is an immediate consequence of the following lemma

\begin{lemma}\label{tl3} Suppose $0<a< b < \infty, \hat{f}\in L^{\infty}(\mathbb R), \  \text{supp} \ \hat{f} \subset \{ \xi : a < |\xi|< b \},$  then 
\[ \Sigma(\xi)= \sum_{k, m \in \mathbb Z}  |\hat{f}(\xi+ \beta^{-1}(k+1/2)-\alpha m)| |\hat{f}(\xi-\alpha m)| \lesssim \|\hat{f}\|_{L^{\infty}},\]
for almost every $\xi \in \mathbb R.$ 
 \end{lemma}
 \begin{proof}
Since the proof is similar to that of  Lemma \ref{tl2} we will omit it.
\end{proof}

\begin{proof}[\bf Proof of Lemma~\ref{tl1}]
Lemma \ref{tl1} follows from the   observations we made above, together with  Lemmas \ref{tl2} and \ref{tl3}.
\end{proof}

\begin{remark}\label{r2} As a consequence of  Lemma \ref{tl1}, we may conclude that $\Phi_k, \Delta_k \in L^1_{loc}(\mathbb R).$ Indeed,
\begin{enumerate}
\item  Taking  $\hat{f}= \chi_K$, where $K\subset \mathbb R$ is any compact set, and  fixing $k_0\in \mathbb Z,$   by Lemma \ref{tl1} we obtain
\begin{eqnarray*}
  \int_{K}\left( \sum_{m\in \mathbb Z} |\hat{\phi}(\xi-\alpha m)|   \left|\overline{\hat{\phi}(\xi+ \beta^{-1}k_0-\alpha m)} \right|\right)   \chi_{K}(\xi+\beta^{-1}k_0)d\xi < \infty.
\end{eqnarray*}
It follows that ${\overline{\Phi_{-k_0}}} \in L^1_{loc}(\mathbb R),$ and so $\Phi_{k_0}\in L^{1}_{loc}(\mathbb R).$

\item  Taking  $\hat{f}= \chi_K$, where $K\subset \mathbb R$ is any compact set, and  fixing $k_0\in \mathbb Z,$   by Lemma \ref{tl1} we obtain
\begin{eqnarray*}
\int_{K}   \left( \sum_{m\in \mathbb Z} |\hat{\phi}(\xi+\alpha m)|  \left|\overline{\hat{\phi}(\xi +\beta^{-1}( k_0 +1/2)-\alpha m)} 
 \right|\right) \chi_{K}(\xi +\beta^{-1}( k_0  +1/2))d\xi < \infty.\nonumber
\end{eqnarray*}
It follows that   $\Delta_{k_0}(\xi -\beta^{1}(k_0+1/2)) \in L^1_{loc}(\mathbb R),$ and so   $\Delta_{k} \in L^{1}_{loc}(\mathbb R).$
\end{enumerate}
\end{remark}

We are now ready to prove Theorem~\ref{chwb}.

\begin{proof}[Proof of Theorem~\ref{chwb}] {(\bf $(b) \implies (a)$}).
Assume that 
$\Phi_k(\xi)= \delta_{k,0}$ a.e. for each $k\in \mathbb Z,$
and $\Delta_k(\xi)=0$  a.e. for  each $k\in \mathbb Z$. Then by Proposition \ref{dp}, it follows that $$\|f\|^2_{L^2}=\sum_{m\in \mathbb N_0} \sum_{j\in \mathbb Z} |\langle f, \psi_{j,m}  \rangle|^2 $$
for all $f\in \mathcal{D}.$ By Proposition \ref{KS}\eqref{pl2}, we may conclude  that the above equality holds for all $f\in L^2(\mathbb R).$ This proves that statement (b) implies statement (a).

\noindent  ({\bf $(a) \implies (b)$}).
\noindent Suppose that (a) holds. Therefore, by Proposition \ref{dp} we have 

\begin{eqnarray}\label{rc1}
\|f\|_{L^2}^2= \sum_{j, m}|\langle f, \psi_{j,m}\rangle|^2= \int_{\mathbb R} \sum_{k\in \mathbb Z} \hat{f}(\xi+\beta^{-1}k)\overline{\hat{f}(\xi)} \Phi_k(\xi) d\xi+\mathcal{I}_1(f)
\end{eqnarray}
for all $f\in \mathcal{D}.$

Let $\xi_0 \in \mathbb R\setminus \mathbb Z.$  Choose $\epsilon>0$ so that $B_{\epsilon}(\xi_0)\cap \mathbb Z=(\xi_{0}-\epsilon, \xi_0+\epsilon)\cap \mathbb Z= \emptyset,$ and set $\hat{f}= \chi_{B_{\epsilon}(\xi_0)}.$  Then for $\xi\in B_{\epsilon}(\xi_0),$ by \eqref{rc1} , we have $\Phi_0(\xi)= 1.$ Since $\xi_0$ is arbitrary, we have  $\Phi_0=1$ a.e.. Since $\Phi_0=1$ a.e., \eqref{rc1} gives
\begin{equation}\label{ssp}
 0 = \int_{\mathbb R} \sum_{0\neq k\in \mathbb Z} \hat{f}(\xi+\beta^{-1}k)\overline{\hat{f}(\xi)} \Phi_k(\xi) d\xi  + \int_{\mathbb R} \sum_{k\in \mathbb Z}\overline{\hat{f}(\xi)} \hat{f}(\xi+\beta^{-1}( k +1/2)) \Delta_k(\xi) d\xi.
\end{equation}
We claim that $\Phi_k=0$ a.e. for all $0\neq k \in \mathbb Z$ and $\Delta_k=0$  a.e. for all $k\in \mathbb Z$.

By a polarization argument (see e.g. \cite[p.362, Section 7.1]{hw}) of~\eqref{ssp}   we obtain
\begin{equation}\label{api}
0= \int_{\mathbb R} \sum_{0\neq k\in \mathbb Z} \hat{g}(\xi+\beta^{-1}k)\overline{\hat{f}(\xi)} \Phi_k(\xi) d\xi + \int_{\mathbb R} \sum_{k\in \mathbb Z}\overline{\hat{f}(\xi)} \hat{g}(\xi+\beta^{-1}( k +1/2)) \Delta_k(\xi) d\xi
\end{equation}
for all $f,g \in \mathcal{D}.$

Let us fix  $k_0\neq 0$ and choose a point $\xi_0$ of differentiability of the integral of $\Phi_{k_0}$ such that  $0\neq \xi_0\neq \xi_0 +\beta^{-1}k_0.$  By Remark \ref{r2}, we have  $\Phi_{k_0}\in L^1_{loc}(\mathbb R)$. Hence, almost every point of $\mathbb R$ is point of differentiability 
of the integral of $\Phi_{k_0}.$  This means, if $\xi_0$ is such a point, by Lebesgue differentiation theorem, we have

\begin{eqnarray}\label{aldt}
\lim_{\delta\to 0}\frac{1}{\mu(B_{\delta}(\xi_0))} \int_{\mathbb R}  \Phi_{k_{0}}(\xi) d\xi= \Phi_{k_0}(\xi_0).
\end{eqnarray}

We  consider $\delta>0$ sufficiently small so that both $B_{\delta}(\xi_0)$ and $B_{\delta}(\xi_0+ k_0)$ lie within $\mathbb{R}\setminus \{0\}$. Let $f_{\delta}$ and $g_{\delta}$ in $\mathcal{D}$ be functions such that
\[ \hat{f}_{\delta}(\xi)=\frac{1}{\sqrt{\mu(B_{\delta}(\xi_0))}} \chi_{B_{\delta}(\xi_0)}(\xi),\]
and
\[ \hat{g}_{\delta}(\xi)=\frac{1}{\sqrt{\mu(B_{\delta}(\xi_0))}} \chi_{B_{\delta}(\xi_0+\beta^{-1}k_0)}(\xi).\]
Note that $\hat{g}_{\delta}(\xi)=\hat{f}_{\delta}(\xi-\beta^{-1}k_0)$ and
\begin{eqnarray}\label{ob1}
\overline{\hat{f}_{\delta}(\xi)}\hat{g}_{\delta}(\xi+\beta^{-1}k_0)=\frac{1}{\mu(B_{\delta}(\xi_0))} \chi_{B_{\delta}(\xi_0)}(\xi).
\end{eqnarray} 
Substituting   $f_{\delta}, g_{\delta}$ in~\eqref{api}, and using~\eqref{ob1}, we obtain
\begin{eqnarray*}
0 & = & \int_{\mathbb R}  \hat{g}_{\delta}(\xi+\beta^{-1}k_0)\overline{\hat{f}_{\delta}(\xi)} \Phi_{k_{0}}(\xi) d\xi + \int_{\mathbb R} \sum_{ k\neq 0, k_0} \hat{g}_{\delta}(\xi+\beta^{-1}k)\overline{\hat{f}_{\delta}(\xi)} \Phi_k(\xi) d\xi\\
& = & \frac{1}{\mu(B_{\delta}(\xi_0))} \int_{B_{\delta}(\xi_0)}  \Phi_{k_{0}}(\xi) d\xi + \int_{\mathbb R} \sum_{ k\neq 0, k_0} \hat{g}_{\delta}(\xi+\beta^{-1}k)\overline{\hat{f}_{\delta}(\xi)} \Phi_k(\xi) d\xi\\
&& + \int_{\mathbb R} \sum_{k\in \mathbb Z}\overline{\hat{f}_{\delta}(\xi)} \hat{g}(\xi+\beta^{-1}( k +1/2)) \Delta_k(\xi) d\xi\\
& = & \frac{1}{\mu(B_{\delta}(\xi_0))} \int_{B_{\delta}(\xi_0)}  \Phi_{k_{0}}(\xi) d\xi + J_{\delta}+P_{\delta}.
\end{eqnarray*}
By~\eqref{aldt}, to establish  that $\Phi_{k_0}(\xi_0)=0$, it suffices to prove that \[\lim_{\delta\to 0} J_{\delta}=\lim_{\delta \to 0} P_{\delta}=0.\]
Assume that  $\overline{\hat{f}_{\delta}}(\xi)\hat{g}_{\delta}(\xi +\beta^{-1}k)=\overline{\hat{f}_{\delta}}(\xi)\hat{f}_{\delta}(\xi +\beta^{-1}(k-k_0))\neq 0,$ for some $k\neq k_0$.  Then $|\xi-\xi_0|< \delta$ and $|\xi+\beta^{-1}(k-k_0)-\xi_0|< \delta.$ But this implies we have 
\[|\beta^{-1}(k-k_0)|=|\xi+ \beta^{-1}(k-k_0)-\xi_0- \xi+\xi_0|\leq 2\delta.\]
Taking $\delta\to 0$, we obtain $k=k_0$ which is  a contradiction. Therefore
$\overline{\hat{f}_{\delta}}(\xi)\hat{g}_{\delta}(\xi +\beta^{-1}k)=0$ for all $k\neq k_0.$ It follows that $J_{\delta}\to 0$ as $\delta \to 0.$

Suppose that for some $k\neq k_0,$ we have $\overline{\hat{f}_{\delta}}(\xi)\hat{g}_{\delta}(\xi +\beta^{-1}(k+1/2))=\overline{\hat{f}_{\delta}}(\xi)\hat{f}_{\delta}(\xi +\beta^{-1}(k-k_0)+(2\beta)^{-1})\neq 0.$  Then $|\xi-\xi_0|< \delta$ and $|\xi+\beta^{-1}(k-k_0) + (2\beta)^{-1}-\xi_0|< \delta.$ But this implies we have 
\[|\beta^{-1}(k-k_0)|\leq |\beta^{-1}(k-k_0)+(2\beta)^{-1}| = |\xi+ \beta^{-1}(k-k_0)+(2\beta)^{-1}-\xi_0- \xi+\xi_0|\leq 2\delta.\]
Taking $\delta\to 0$, we get  $k=k_0,$ which is a contradiction.

Next, assume that  for $k=k_0$,  $\overline{\hat{f}_{\delta}}(\xi)\hat{g}_{\delta}(\xi +\beta^{-1}(k+1/2))=\overline{\hat{f}_{\delta}}(\xi)\hat{f}_{\delta}(\xi +(2\beta)^{-1})\neq 0. $  Then $|\xi-\xi_0|< \delta$ and $|\xi+ (2\beta)^{-1}-\xi_0|< \delta.$ But this implies we have 
\[|(2\beta)^{-1}| = |\xi+(2\beta)^{-1}-\xi_0- \xi+\xi_0|\leq 2\delta.\]
Taking $\delta\to 0$, we get   a contradiction. Therefore
$\overline{\hat{f}_{\delta}}(\xi)\hat{g}_{\delta}(\xi +\beta^{-1}(k+1/2))=0$ for all $k\in \mathbb Z.$ It follows that $P_{\delta}\to 0$ as $\delta \to 0.$ Since $k_0$ is arbitrary,  we have $\Phi_{k}(\xi)=0$ for $0\neq k \in \mathbb Z.$

The proof that  $\Delta_k=0$  a.e. for all $k\in \mathbb Z$ is similar to the above using the functions $f_{\delta}$ and $g_{\delta}$ in $\mathcal{D}$ be defined by 
\[ \hat{f}_{\delta}(\xi)=\frac{1}{\sqrt{\mu (B_{\delta}(\xi_0))}} \chi_{B_{\delta}(\xi_0)}(\xi),\]
and
\[ \hat{g}_{\delta}(\xi)=\frac{1}{\sqrt{\mu (B_{\delta}(\xi_0))}} \chi_{B_{\delta}(\xi_0+\beta^{-1}(k_0+1/2))}(\xi).\]
\end{proof}

We can now prove Corollary~\ref{cor-par-onb}

\begin{proof}[Proof of Corollary~\ref{cor-par-onb}
] Suppose that  $$\begin{cases}\|\phi\|_{L^2}= \frac{1}{\sqrt{2\beta}}\\ \Re \langle X_{j,m}, Y_{j,m}\rangle=0\end{cases}$$  for all $(j, m) \in \mathbb Z \times \mathbb N.$ Then, we have  $\|\psi_{j,0}\|_{L^2}= \sqrt{2\beta}\|\phi\|_{L^2}=1$ for $j\in \mathbb Z,$ and 
\begin{eqnarray*}
\|\psi_{j,m}\|^2_{L^2} & = &  \|X_{j,m}\|^2_{L^2}+\|Y_{j,m}\|^2_{L^2}+ 2\Re \langle X_{j,m}, Y_{j,m}\rangle \\
& = & 2\beta \|\phi\|_{L^2}^2 + 2\Re \langle X_{j,m}, Y_{j,m}\rangle =1. 
\end{eqnarray*}
The converse easily follows. 

\end{proof}

\section{Parseval Wilson frames}\label{cgfwb} 
In this section we connect Gabor tight frames to the Wilson systems we defined. In particular, one of our main result is Theorem~\ref{wg} from which Theorem~\ref{thm-main1} follows.

\subsection{From tight Gabor frames to Parseval Wilson frames}
We can now state and prove a result that links Gabor frames to the Wilson systems defined in the Introduction.

\begin{theorem} \label{wg}  Let $\phi \in L^2(\mathbb R)$ and $\alpha, \beta >0$. The following two statements are equivalent.

\begin{enumerate}
\item[(a)] The  Gabor system  $\mathcal{G}(\phi, \alpha, \beta)$ is a  tight frame for $L^2(\mathbb R)$   with frame bound  $\beta^{-1}$, and $\Delta_k=0$  a.e. for all $k\in  \mathbb Z$, where $\Delta_k$ was defined in Theorem~\ref{chwb}. 
\item[(b)] The Wilson system $\mathcal{W}(\phi, \alpha, \beta)$ is a Parseval  frame for $L^2(\mathbb R)$.
\end{enumerate}
%
\end{theorem}

\begin{proof}[Proof of Theorem~\ref{wg}]
{\bf $((a)\implies (b)).$} 
Assume that (a) holds.  By Proposition~\ref{wcza}, if $\mathcal{G}(\phi, \alpha, \beta)$  is a tight frame with frame bound $\beta^{-1},$   then $\Phi_k(\xi)=0$ a.e. for all $k\in \mathbb Z$. Together with the second condition of (a) we conclude using Theorem~\ref{chwb} that (b) holds.

\noindent {\bf $((b)\implies (a)).$} The converse follows from Theorem~\ref{chwb}, and Proposition~\ref{wcza}.
\end{proof}
The following consequence easily follows from Theorem~\ref{wg}.

\begin{corollary}\label{wonb}
Let $\phi \in L^2(\mathbb R)$ and $\alpha, \beta >0$. Let $X_{j,m}$ and $Y_{j,m}$ be defined by 
\begin{equation*}
\begin{cases}
X_{j,m}=e^{-2\pi i \beta j \alpha m}\phi_{j,m},\\
Y_{j,m}= (-1)^{j+m}e^{2\pi i \beta j \alpha m} \phi_{j,-m}.
\end{cases}
\end{equation*} Suppose that the  Gabor system  $\mathcal{G}(\phi, \alpha, \beta)$ is a  tight frame for $L^2(\mathbb R)$   with frame bound  $\beta^{-1}$, and $\Delta_k=0$  a.e. for all $k\in  \mathbb Z.$   Then, the Wilson system $\mathcal{W}(\phi, \alpha, \beta)$ is an orthonormal basis for $L^2(\mathbb R)$ if and only if $$\begin{cases}\|\phi\|_{L^2}= \frac{1}{\sqrt{2\beta}}\\ \Re \langle X_{j,m}, Y_{j,m}\rangle=0\end{cases}$$  for all $(j, m) \in \mathbb Z \times \mathbb N.$

\end{corollary}

\begin{proof}

The proof  follows from Theorem~\ref{wg} and Corollary~\ref{cor-par-onb}. 

\end{proof}

\begin{remark}\label{FR} Here are some observations from Theorem \ref{wg}.
\begin{enumerate}
\item \label{FR2} Suppose  that $\alpha=1$ and $\beta=\tfrac{1}{2n}$ where $n$ is any odd natural number. If we assume that $\hat{\phi}$ is a real-valued function,  then $\Delta_k(\xi)=0$ is automatically satisfied.  Indeed, in this case, by a change of variable ($m\mapsto 2k+1 -m$) over summation, we obtain $\Delta_k(\xi)= - \Delta_k(\xi),$ that  is, $\Delta_k(\xi)=0.$

\item Suppose that  $\alpha =1$ and $\beta^{-1} \in \mathbb N$  and $\hat{\phi}$ is real valued. Then  we shall construct a generator $\phi$ of Wilson system in Theorem \ref{wg}, using the Zak transform, see  Section \ref{ztrc}.

\end{enumerate}
\end{remark}

We conclude this section by   stating an analogue of Theorem~\ref{wg} in higher dimensions.  Since  the proofs are identical with the obvious modifications, we omit them.

 To state these results we need the following notations. Put $\mathbb N_0^d= \{0\}\cup \mathbb N^d, $ and  $1/\bar{2}=(1/2,\cdots, 1/2)\in \mathbb R^d.$  Let $a= (a_1,\cdots, a_d), b=(b_1, \cdots, b_d) \in \mathbb R^d.$  Let $A$ and $B$ be diagonal
matrices with diagonal entries  $\{ a_1, ...., a_d\}$ and $\{b_1,... , b_d\}$ respectively. Assume that $\det B \neq 0. $ Then  $B^{-1}= \text{diag} \{ b_1, ..., b_d\},$ and  put $b^*= |\det B|.$  For  $m=(m_1,..., m_d) \in \mathbb Z^d, Am= (m_1a_1,..., m_da_d)$   as usual. 
Let  $\phi:\mathbb R^d \to \mathbb C$ be a nice function.
We consider the multivariate Gabor system $$\mathcal{G}(\phi, A, B)=\{ \phi_{j,m}\}_{j, m \in \mathbb Z^d}= \{
\phi_{j,m}(x)= e^{2\pi i x \cdot Am}\phi(x-Bj)\}_{j, m \in \mathbb Z^d}$$ 
We define a  family generated by arbitrary  time-frequency shifts
\begin{align*}
\psi_{j,m}= \begin{cases} \phi (x-Bj) & \text{if} \  m=0, j \in \mathbb Z^d \\ 
 e^{-2\pi i Bj\cdot Am}\phi_{j,m}(x) + e^{2\pi i Bj\cdot Am} \phi_{j,-m}(x) & \text{if} \  j\in \mathbb {Z}^d,  0\neq m \in \mathbb N_0^d. 
\end{cases}
\end{align*}
The collections of these functions  is denoted by 
\[\mathcal{W}(\phi, A, B)= \{ \psi_{j,m}: j \in \mathbb Z^d, m \in \mathbb N_0^d\}.\]
We call  $ \mathcal{W}(\phi, A, B)$ the Wilson system. 
Specifically, we have  following result.

\begin{theorem}\label{chwb-d}The following statements are equivalent:
\begin{enumerate}
\item[(a)] The  Gabor system 
\[ \mathcal{G}(\phi, A, B)= \left \{e^{2\pi i x \cdot Am}\phi(x-Bj): m, j \in \mathbb Z^d\right \}\]
is a  tight frame for $L^2(\mathbb R^d)$ with frame bound  $(\det B)^{-1}$, and $\Delta_k^d=0$ a.e. for all $k\in \mathbb N_0^d,$ where $$\Delta_k^d(\xi):= \sum_{0\neq m \in \mathbb Z^d} (-1)^m\hat{\phi}(\xi+Am)\overline{\hat{\phi}(\xi+B^{-1}(k +1/\bar{2})-Am)}.$$
\item[(b)]  The Wilson system $\mathcal{W}(\phi, A, B)$ is a Parseval frame for  $L^2(\mathbb R^d)$.
\end{enumerate}


\end{theorem}

\subsection{Examples  of  generator of Wilson systems} \label{cwlf} 

In this subsection we prove that  there exists rapidly decaying   $C^{\infty}$ function $\phi$  satisfying the hypothesis of Theorem~\ref{wg}. Thus we seek a function $\phi\in L^2(\R)$ which satisfies 
 $$\Phi_k(\xi)=\sum_{ m \in \mathbb Z} \hat{\phi}(\xi-\alpha m) \overline{\hat{\phi}(\xi+\beta^{-1}k-\alpha m)} =\delta_{k,0}  \ a.e  \   \text{for each} \ k \in \mathbb Z, $$  
 $$ \Delta_k(\xi)=\sum_{m \in \mathbb Z} (-1)^m\hat{\phi}(\xi+\alpha m)\overline{\hat{\phi}(\xi+\beta^{-1}(k +1/2)-\alpha m)}=0 \ a.e  \   \text{for each}  \ k \in \mathbb Z.$$

We give two classes of examples, one when $\alpha\beta=1/2$, which is the classical case developed in \cite{djl, auscher1994remarks}. The second family of examples concerns the case $\beta\in (0, 1/2)$ and $\alpha=1$. 
 
\begin{example}\label{example1} In this example we assume $\alpha\beta=1/2$ and recovers the classical case. Define $\hat{\phi}=\chi_{[0, \alpha]}$. 
We note that $\|\phi\|_{L^2}^2=\|\hat{\phi}\|_{L^2}^2=\frac{1}{2\beta},$ and $\hat{\phi}$ is supported in $[0, 1/2\beta].$  Since $\Phi_k(\xi)$ is periodic with period $\alpha,$ 
 we only needs to check what happens for $0\leq \xi \leq \alpha.$  Since the support of $\hat{\phi}$ is $[0, \alpha],$ we have  $\Phi_0 =1 $ a.e., $\Phi_k =0$ a.e. for $k\neq 0,$ and  $\Delta_k=0$ a.e. for all $k\in \mathbb Z.$ We also note that $\hat{\phi}(\xi) \overline{\hat{\phi}(\xi + 2 \alpha m)}=0$ for all $\xi \in \mathbb R$  and all $m\in \mathbb N.$  
On the other hand, by the Plancherel theorem  and \eqref{fabtf}, we have 
\begin{eqnarray}\label{rp}
\langle X_{j,m}, Y_{j,m}\rangle  & =  & \langle \widehat{X_{j,m}}, \widehat{Y_{j,m}} \rangle  
 =  (-1)^{j+m} \int_{\mathbb R} \hat{\phi}(\xi) \overline{\hat{\phi}(\xi + 2 \alpha m)} d\xi.
\end{eqnarray}
Hence by \eqref{rp},   it follows that  $\text{Re} \langle X_{j,m}, Y_{j,m} \rangle =0$ for all $(j,m) \in \mathbb Z \times \mathbb N.$   Thus, this example satisfies all the  hypotheses of Theorem~\ref{wg}.
\end{example}

\begin{example}\label{example2} 

Let $\beta\in (0, 1/2),$ and $\alpha=1.$
 For this case, we choose a function  $\hat{\phi}:\mathbb R \to \mathbb C$ supported in  $B_{\gamma}(0)=\{ \xi \in \mathbb R: |\xi|\leq \gamma \},$ where $\gamma = \frac{1}{4\beta}-\epsilon$ for $\epsilon>0$ suitable small enough so that $1<2\gamma,$ that is, $1< \frac{1}{2\beta}-2\epsilon.$  
 
 We note that for $\beta\in (0,1/2),$ we have $1<1/2\beta,$ and hence we may choose $\epsilon>0$ so that  $1< \frac{1}{2\beta}-2\epsilon.$ (For fix $\beta>0,$ take $2\epsilon= \frac{1}{2\beta}-1-\epsilon'$ for suitable small $\epsilon'>0$ and notice that for this choice of $\epsilon,$ we have $\gamma<1.$)

 For this $\hat{\phi},$ we note that
 $$\hat{\phi}(\xi) \overline{\hat{\phi}(\xi+\beta^{-1}k)}=0$$
 for all $0\neq k \in \mathbb Z,$ and
 $$\hat{\phi}(\xi) \overline{\hat{\phi}(\xi+\beta^{-1}(k+1/2))}=0$$
 for all $k \in \mathbb Z.$ In fact, if possible, assume that $k\neq 0$ and  $\hat{\phi}(\xi) \overline{\hat{\phi}(\xi+\beta^{-1}k)}\neq 0,$ then $|\xi| < \gamma$ and $|\xi + \beta^{-1}k| <\gamma.$ But this implies we have
 \[|\beta^{-1}k|=|-\beta^{-1}k|=|-\beta^{-1}k-\xi + \xi|< 2\gamma.\]
  Since $\gamma<\frac{1}{4\beta}$, we have  $|k|<2\beta\gamma < 1/2,$ therefore   we must have $k=0$, which is a contradiction.
 In fact, if possible,  assume that $\hat{\phi}(\xi) \overline{\hat{\phi}(\xi+\beta^{-1}(k+1/2))} \neq 0,$ then $|\xi|< \gamma$ and $ |\xi+\beta^{-1}(k+1/2))|< \gamma.$  But this implies 
 \[ |\beta^{-1}(k+1/2)|=|-\beta^{-1}(k+1/2)-\xi + \xi| \leq  2\gamma.\]
 Since $\gamma<\frac{1}{4\beta}$, we have  $|k+1/2|\leq 2\beta \gamma< 1/2$ but this is not possible as $k\in \mathbb Z.$
 Thus, we have $\Phi_k=0$ a.e. for all $0\neq  k\in \mathbb Z$ and $\Psi_k=0$ a.e. for all $k\in \mathbb Z.$
  
 Next, we wish to show that  $\Phi_0(\xi)= \sum_{m\in \mathbb Z} |\hat{\phi}(\xi-m)|^2=1$ a.e.. 
 Since this sum is periodic in $\xi$ with period 1, we only needs to check what happen for $0\leq \xi \leq 1.$
 To this end,  consider smooth function $G:\mathbb R \to [0,1]$ satisfying the following properties:
\[G(x)=\begin{cases} 0 &  \text{if} \ \  x \leq -\gamma+1,\\ 
1 & \text{if} \  \  x\geq \gamma.
\end{cases}
\]
We define the function $\hat{\phi}:\mathbb R \to \mathbb R$  by
 \[ \hat{\phi}(\xi)= \begin{cases} \sin \left[ \frac{\pi}{2} G(\xi+1) \right] &  \text{if} \ \xi\leq 0,\\
 \cos \left( \frac{\pi}{2} G(\xi) \right) & \text{if} \ \xi\geq 0.
 \end{cases}
 \]
 We note that $\hat{\phi}$ is supported in $[-\gamma, \gamma].$

Since $\hat{\phi}$ is supported in  $B_{\gamma}(0)\subset [-1, 1],$ it follows that 
$\hat{\phi}(\xi) \overline{\hat{\phi}(\xi + 2  m)} =0$ for all $m\in \mathbb N.$ In fact, if  $\hat{\phi}(\xi) \overline{\hat{\phi}(\xi + 2  m)} \neq 0,$ then  $|\xi|< \gamma $ and $ |\xi + 2m| < \gamma.$ But this implies we have $|2m|< |\xi + 2m| + |\xi| <2\gamma,$ and so $|m|\leq \gamma,$ which is contradiction  as $\gamma<1.$  Thus, for  real $\hat{\phi} $ with  support in $B_{\gamma}(0)\subset [-1,1],$ to show $\Phi_0=1,$ we only need to ascertain that $ \hat{\phi}^2(\xi)+\hat{\phi}^2(\xi-1)$ for all $0\leq \xi \leq 1.$ This is easy to verify. For  the above defined $\hat{\phi},$ we have 
\[\hat{\phi}^2(\xi)+\hat{\phi}^2(\xi-1)= \cos ^2 \left( \frac{\pi}{2}G(\xi) \right)+\sin ^2 \left( \frac{\pi}{2}G(\xi) \right) =1, \   \ (\xi \in [0,1]).\]

Since $\hat{\phi} \in C^{\infty}_c(\mathbb R),$ we have $\phi \in \mathcal{S}(\mathbb R)$. We note that this $\phi$ satisfies the hypothesis of  Theorems  \ref{chwb} and \ref{wg}. 

\end{example}

\begin{remark}\label{noonbs}
The Parseval Wilson frames of Example~\ref{example2} cannot lead to an ONB. Indeed, in order to have an ONB one must also choose $\phi$ so that 
$\|\phi\|_2=1/\sqrt{2\beta}$  and  $ \text{Re} \langle X_{j,m}, Y_{j,m}  \rangle =0 \ \text{for all} \ (j,m) \in \mathbb Z \times \mathbb N .$   However, given a function $\phi \in \mathcal{S}(\R)$ constructed in Example~\ref{example2}, we note that 
$\hat{\phi}$ is supported in $[-C, C]$ with $1/2<C< 1$ and 
\begin{eqnarray*}
\|\phi\|_{L^2}^2 & = & \|\hat{\phi}\|^2_{L^2} = \int_{-C}^{C} |\hat{\phi}(\xi)|^2 d\xi=  \int_{-1}^{1} |\hat{\phi} (\xi)|^2 d\xi =  \int_0^1 |\hat{\phi}(\xi)|^2+|\hat{\phi}(\xi-1)|^2 d\xi=1
\end{eqnarray*}
which happens only when $\beta=1/2$. 
\end{remark}

We can now prove Theorem~\ref{thm-main1}

\begin{proof}[\textbf{Proof of Theorem~\ref{thm-main1}}]
Choose  $\phi$ as in  Example \ref{example2}.
\end{proof}

\section{The Zak transform and Wilson systems }\label{ztrc}
In this section we construct example of generators $\phi$ that satisfy the hypothesis of Theorem~\ref{wg} and such that $\phi$ and $\hat{\phi}$ have exponential decay. To achieve this we extended a construction originally given in \cite{djl} to the case of Gabor frame of redundancy $N\in \N$ when $N\geq 3$. The key tool needed to deal with this case is the  Zak transform.  Using this we have the following results.

\begin{theorem}\label{hp}  Let $\hat{\phi}$ be real functions such that $|\hat{\phi}(\xi)|\lesssim  (1+ |\xi|)^{-1-\epsilon}$ and  $\beta=1/(2n)$ where $n$ is any odd natural number.  Then the following are equivalent:
\begin{enumerate}
\item  The  Gabor system  $\mathcal{G}(\phi, 1, \beta)$ is a  tight frame for $L^2(\mathbb R)$   with frame bound  $\beta^{-1}$. 
\item  The Wilson system $\mathcal{W}(\phi, 1, \beta)$ is a Parseval  frame for $L^2(\mathbb R)$. 
\item The Zak transform $Z_{\beta} \hat{\phi}$ of $\hat{\phi}$ satisfies
\begin{eqnarray*}
  \sum_{r=0}^{\beta^{-1}-1} \left| Z_{\beta} \hat{\phi}\left(x,   \xi- \beta r \right)\right|^2= \frac{1}{\beta}  
\end{eqnarray*}
for all most all $x, \xi \in [0,1].$
\end{enumerate}

Furthermore, if one of the above statements holds (hence all of them hold) , then the Parseval Wilson frame $\mathcal{W}(\phi, 1, \beta)$ is an orthonormal basis for $L^2(\mathbb R)$ if and only if   $\text{Re} \langle X_{j,m}, Y_{j,m} \rangle =0$ for all  $(j,m)\in \mathbb Z \times \mathbb N$ and $\|\phi\|_{L^2}= 1/\sqrt{2\beta}$, where $X_{j,m}$ and $Y_{j,m}$ were defined in Theorem~\ref{wg}.
\end{theorem}

We shall prove the above theorems at the end of  the section. To this end,   we first  develop some tools using Zak transform.  In particular, this framework will allow us to convert  the infinitely many conditions  Theorem~\ref{chwb} (b) (one for every  $k$) into  a  single condition which can be tested (see Proposition \ref{ciz} below). Thus, we show how to use the Zak transform to construct smooth functions that satisfy the hypotheses of Theorem~\ref{chwb} and Theorem~\ref{wg} (see Theorem \ref{ndua} below).

Given $\beta>0$, we define the Zak transform of $f\in \mathcal{S}(\R)$ by 
\begin{equation}\label{ZT}
Z_{\beta}f(x,\xi)=  \frac{1}{\sqrt{\beta}} \sum_{k\in \mathbb Z}  f(\beta^{-1} (\xi-k)) e^{2\pi i k x }. 
\end{equation}
The  two-variable function $F=Z_{\beta} f$ is periodic in the first variable and ``semi-periodic" in the second variable:
\begin{eqnarray}\label{qp}
Z_{\beta}f(x+1, \xi) =Z_{\beta} f(x, \xi), \  \  Z_{\beta}f(x, \xi \pm1)= e^{\pm 2\pi i x} Z_{\beta} f(x, \xi).
\end{eqnarray}
 The set of all functions $F$ of two variables satisfying the periodicity conditions  \eqref{qp} can be equipped with the norm
\begin{eqnarray}\label{zn}
\|F\|^2= \int_0^1 \int_0^1 |F(x, \xi)|^2 dx d\xi.
\end{eqnarray} 
 We will denote the closure of this set, under the norm  \eqref{zn}, by $\mathcal{Z}.$ A function $F$ is in $\mathcal{Z}$ if and only if its restriction to $[0, 1)\times [0,1)$ is square integrable
 and it satisfies the periodicity  conditions almost everywhere. It  follows that $ \mathcal{Z}$ is isomorphic with $L^2([0,1)^2)$ and  the map   $Z_{\beta}$ defined by  \eqref{ZT}  can be extended to unitary map from $L^2(\R)$ to  $\mathcal{Z}:$
 \begin{eqnarray}\label{uzt}
 \int_{[0,1]^2} |Z_{\beta}f(x, \xi)|^2 dx d\xi = \|f\|_{L^2}^2.
 \end{eqnarray}
The functions  $E_{m,n} (x, \xi),$ defined by 
 \[  E_{m,n} (x, \xi) = e^{2\pi i nx} e^{2\pi i m\xi}  \   \ \text{for} \ x, \xi \in [0,1)\] 
constitute an orthonormal basis for $ \mathcal{Z}.$
 Let $\phi\in \mathcal{S}(\R).$ The inverse transform of \eqref{ZT} is given by 
\begin{eqnarray}\label{ifz}
\phi(\xi) = \sqrt{\beta}\int_0^1 Z_{\beta} \phi (x, \beta \xi) dx.
\end{eqnarray}

\begin{lemma} \label{rzf}Let $\phi \in \mathcal{S}(\R)$ and $\beta^{-1} \in \mathbb N.$ Then we have  $$Z_{\beta}\phi(x, \xi) =  \beta  e^{2\pi i \xi x} \sum_{j=0}^{\beta^{-2}-1} e^{2\pi i \xi j} Z_{\beta} \hat{\phi} \left(-\beta^{-2}\xi, \frac{x+j}{\beta^{-2}}\right)$$
and 
$$Z_{\beta} \hat{\phi}(x, \xi) =  \beta  e^{2\pi i \xi x} \sum_{j=0}^{\beta^{-2}-1} e^{2\pi i \xi j} Z_{\beta} \phi \left(\beta^{-2}\xi, -\frac{x+j}{\beta^{-2}}\right).$$
  \end{lemma}

  \begin{proof}
  Denote $ T_xf(t)=f(t-x),  M_{\xi}f(t)= e^{2\pi i \xi t}f(t).$  For fixed $\xi$ and $\beta,$ put $h(t)=\phi (\beta^{-1} (\xi-t)) \ (t\in \R).$  Then  $\hat{h}(y)=\beta e^{-2\pi i \xi y} \hat{\phi}(-\beta y) \ (y\in \R).$ Using the Poisson summation formula (see e.g., \cite[p.16 (1.35)]{gro}), we find 
\begin{eqnarray*}
Z_{\beta}\phi (x,\xi) & = &  \frac{1}{\sqrt{\beta}} \sum_{k\in \mathbb Z}  \phi (\beta^{-1} (\xi-k)) e^{2\pi i k x }= \frac{1}{\sqrt{\beta}} \sum_{k\in \mathbb Z} h(k)e^{2\pi i k x } \ \  \  \\
& = &  \frac{1}{\sqrt{\beta}} \sum_{k\in \mathbb Z} \widehat{(M_xh)}(k) = \frac{1}{\sqrt{\beta}} \sum_{k\in \mathbb Z} T_x\widehat{h}(k) \\
& = &  \sqrt{\beta} e^{2\pi i \xi x}\sum_{k\in \mathbb Z} \hat{\phi}(-\beta (k-x)) e^{-2\pi i \xi  k} \\
& = &  \sqrt{\beta} e^{2\pi i \xi x}\sum_{k\in \mathbb Z} \hat{\phi}\left(\beta^{-1} \left(  \frac{x-k}{\beta^{-2}}\right) \right) e^{-2\pi i \xi  k}. \\
\end{eqnarray*} 
Noticing  $\mathbb Z = \{ \beta^{-2}k-j: k\in \Z, j=0,1,..., (\beta^{-2}-1)\},$ we may rewrite 
\begin{eqnarray*}
\sum_{k\in \mathbb Z} \hat{\phi}\left(\beta^{-1} \left(  \frac{x-k}{\beta^{-2}}\right) \right) e^{-2\pi i \xi  k} & = & \sum_{j=0}^{\beta^{-2}-1} \sum_{k\in \mathbb Z} \hat{\phi}\left(\beta^{-1} \left(  \frac{x+j}{\beta^{-2}} - k\right) \right) e^{-2\pi i \xi  (\beta^{-2}k -j)}\\
& = & \sqrt{\beta} \sum_{j=0}^{\beta^{-2}-1} e^{2\pi i \xi j} Z_{\beta} \hat{\phi}\left(-\beta^{-2}\xi, \frac{x+j}{\beta^{-2}}\right).
\end{eqnarray*}
This completes the proof of first identity.  Since the second identity can be obtained similarly, we shall omit the details. 
\end{proof}

\begin{proposition} \label{ciz} Let $\phi$ be a real-valued function such that  $\hat{\phi}$ and   $\phi$  have  exponential decay. Suppose that $\alpha =1$ and $\beta^{-1}\in \mathbb N.$
Then 
\[\sum_{ m \in \mathbb Z} \hat{\phi}(\xi- m) \hat{\phi}(\xi+\beta^{-1}k-m) =\delta_{k,0} \  \  \text{a.e. for each}   \ k\in \mathbb Z\]
if and only if the Zak transform   $Z_{\beta} \hat{\phi}$   of $\hat{\phi}$  satisfies  
\begin{eqnarray}\label{dfc}
  \sum_{r=0}^{\beta^{-1}-1} \left| Z_{\beta} \hat{\phi}\left(x,   \xi- \beta r \right)\right|^2= \frac{1}{\beta}  
\end{eqnarray}
for all most all $x, \xi \in [0,1].$
\end{proposition}

\begin{proof} With the assumptions on $\phi$ and $\hat{\phi}$ all the calculations that follow are justified.
Noticing 
$\mathbb Z = \{ \beta^{-1}m+ r: m \in \mathbb Z, r=0,1,..., (\beta^{-1}-1) \}$ and 
using  \eqref{ifz} and \eqref{qp}, we have 
\begin{eqnarray*}
K& := &  \sum_{ m \in \mathbb Z} \hat{\phi}(\xi- m) \hat{\phi}(\xi+\beta^{-1}k- m)\\
& = & \beta  \sum_{m\in \mathbb Z} \int_0^1 \int_0^1  Z_{\beta} \hat{\phi}(x, \beta(\xi-m))  Z_{\beta} \hat{\phi}(x', \beta(\xi-m)+k) dx dx'  \\
& = &  \beta  \sum_{m\in \mathbb Z} \int_0^1 \int_0^1  Z_{\beta} \hat{\phi}\left(x,   \frac{\xi-m}{\beta^{-1}} \right))  Z_{\beta} \hat{\phi}\left(x',  \frac{\xi-m}{\beta^{-1}} +k\right) dx dx'\\
& = &  \beta  \sum_{r=0}^{\beta^{-1}-1} \sum_{m\in \mathbb Z} \int_0^1 \int_0^1  Z_{\beta} \hat{\phi}\left(x,   \frac{\xi-r}{\beta^{-1}}-m \right)  Z_{\beta} \hat{\phi}\left(x',  \frac{\xi-r}{\beta^{-1}} -m+k\right) dx dx'\\
& = &  \beta  \sum_{r=0}^{\beta^{-1}-1} \sum_{m\in \mathbb Z} \int_0^1 \int_0^1  e^{-2\pi i m (x+x')} e^{2\pi i x'k} Z_{\beta} \hat{\phi}\left(x,   \frac{\xi-r}{\beta^{-1}} \right)  Z_{\beta} \hat{\phi}\left(x',  \frac{\xi-r}{\beta^{-1}} \right) dx dx'\\
& = &  \beta  \sum_{r=0}^{\beta^{-1}-1} \sum_{m\in \mathbb Z} \int_0^1  e^{-2\pi i m x} Z_{\beta} \hat{\phi}\left(x,   \frac{\xi-r}{\beta^{-1}} \right) \left( \int_0^1 Z_{\beta} \hat{\phi}\left(x',  \frac{\xi-r}{\beta^{-1}} \right)   e^{2\pi i x'(k-m)}  dx' \right) dx\\
& = &  \beta  \sum_{r=0}^{\beta^{-1}-1}  \int_0^1   Z_{\beta} \hat{\phi}\left(x,   \frac{\xi-r}{\beta^{-1}} \right) \left( \sum_{m\in \mathbb Z} c_{k-m} e^{-2\pi i m x} \right) dx\\
& = &  \beta  \sum_{r=0}^{\beta^{-1}-1}  \int_0^1   Z_{\beta} \hat{\phi}\left(x,   \frac{\xi-r}{\beta^{-1}} \right) \left( \sum_{m\in \mathbb Z} c_{m} e^{2\pi i m x} \right)  e^{-2\pi i kx}dx\\
& = & \beta  \int_0^1 \sum_{r=0}^{\beta^{-1}-1} \left| Z_{\beta} \hat{\phi}\left(x,   \frac{\xi-r}{\beta^{-1}} \right)\right|^2 e^{-2\pi i xk} dx,
\end{eqnarray*}
where  $ c_{k-m}= \int_0^1 Z_{\beta} \hat{\phi}\left(x',  \frac{\xi-r}{\beta^{-1}} \right)   e^{-2\pi i x'(m-k)}  dx'$ is the Fourier coefficient of function $Z_{\beta} \hat{\phi}\left(\cdot,  \frac{\xi-r}{\beta^{-1}} \right)$ at  the point $m-k.$ Hence, the proof follows.
\end{proof}

We  can  construct explicit  ``nice"  $\phi$ that satisfying  hypothesis of Theorem \ref{wg} by constructing $\phi$ satisfying  \eqref{dfc}. The method we used is an extension of the construction given in \cite[Section 4]{djl} for the case $\alpha=1$, $\beta=1/2$.

We start with a real-valued  function $g$ with exponential decay,
\begin{equation}\label{ed}
\begin{cases}
 |g(x)|\leq C e^{-\lambda |x|},   \    \ x\in \mathbb R, \lambda>0,\\
 |\hat{g}(\xi)|\leq C e^{-\mu |\xi|},   \   \    \   \xi \in \R, \mu>0.
 \end{cases}
\end{equation}

The function $g$ will be used as  seed to construct a function in $\mathcal{Z}$ (see \eqref{tpsi} below) that satisfies the condition of Proposition \ref{ciz} \eqref{dfc}.

Observe that  $G:= Z_{\beta}g$ is a well-defined continuous and bounded function. Furthermore, since $g$ is real-valued we have, for $x, \xi \in \R,$
\begin{eqnarray}\label{dg}
G(-x, \xi) = \overline{ G(x, \xi)}.
\end{eqnarray}
Assume further that
\begin{eqnarray}\label{ifc}
\inf_{x, \xi \in [0,1]} \sum_{r=0}^{\beta^{-1}-1}  \left| G \left(x,   \xi- \beta r) \right)\right|^2 >0.
\end{eqnarray}

We then define
\begin{eqnarray}\label{mi}
\hat{\phi}= Z_{\beta}^{-1} \Psi,
\end{eqnarray}
where
\begin{eqnarray}\label{tpsi}
\Psi (x,\xi)= \frac{1}{\sqrt{\beta}}  \frac{G(x, \xi)}{ \left(\sum_{r=0}^{\beta^{-1}-1}  \left| G \left(x, \xi - \beta r \right)\right|^2 \right)^{1/2}}, 
\end{eqnarray}
and
\[ Z_{\beta}^{-1}\Psi (\xi)=  \sqrt{\beta} \int_0^1 \Psi (x, \beta \xi) dx.\]

\begin{theorem} \label{ndua} The function $\hat{\phi},$ defined by \eqref{mi}, is real-valued and satisfies \eqref{dfc}. Furthermore, $\phi$ and $\hat{\phi}$
 have exponential decay.  
 \end{theorem}
  \begin{proof}[Sketch of the Proof]  The detail proof for the case $\alpha=1, \beta=1/2$ can be found in \cite[Theorem 4.1]{djl}.  Since the main ideas for the generalization   is similar, we shall highlight  only the crucial points and omit the details. Now, for the clarity of presentation, we divide the sketch  proof into four steps.

 \noindent \textbf{Step I}: It follows from  \eqref{dg} and \eqref{tpsi} that 
 $\Psi(-x, \xi)= \overline{\Psi (x, \xi)}$
 and so, using \eqref{qp} and \eqref{ifz}, we have $\overline{\hat{\phi}(\xi)}  =  \hat{\phi}(\xi).$

\noindent \textbf{Step II}:  The function $\hat{\phi}$
 has an  exponential decay.   To achieve this, we may follow the procedure:
 \begin{enumerate}
 \item[1.] Because of the decay condition \eqref{ed}, the series 
 \begin{eqnarray*}
G(z, \xi):= G(x+i\tau, \xi)= \frac{1}{\sqrt{\beta}} \sum_{\ell \in \mathbb Z} e^{2\pi i (x+ i \tau ) \ell} g\left(\beta^{-1} (\xi-\ell) \right)
 \end{eqnarray*}
 converges absolutely for $\tau> - \lambda/\pi.$  The extension $G(z, \xi),$ for fixed $\xi \in \mathbb R,$ is complex analytic on $\R + i (-\lambda/\pi, \infty)$   and  satisfies
 \begin{eqnarray*}
 \begin{cases} 
 G(z, \xi +1)= e^{2\pi i z} G(z, \xi)\\
 G(z+1, \xi)= G(z, \xi).
 \end{cases}
 \end{eqnarray*}
 \item[2.] We show that  $\Psi$  (see \eqref{tpsi})  also has analytic extension (the main obstacle  is its denominator). To this end,  we define, for $z\in \R + i (-\lambda/\pi, \infty), \xi\in \R$
 \begin{eqnarray*}
 \mathcal{G}(z, \xi)=  \sum_{r=0}^{\beta^{-1}-1}G(z, \xi -\beta r) G(-z, \xi -\beta r). 
 \end{eqnarray*}
 Then $\mathcal{G}(\cdot, \xi)$ is analytic on $\R + i (-\lambda/ \pi, \infty)$ for every $\xi \in \R,$ and 
 \begin{eqnarray*}\label{ag}
 \mathcal{G}(z+1, \xi) = \mathcal{G}(z, \xi)= \mathcal{G}(z, \xi +1) 
 \end{eqnarray*}
 for all   $z\in \R + i (-\lambda/ \pi, \infty).$ Using \eqref{ed} and  \eqref{ag}, $\mathcal{G}$  is uniformly continuous on  $\R + i [-\lambda/ \pi, \infty)\times \R.$  Because of condition \eqref{ifc},  there exists $ \tilde{\lambda}>0$ so that $|\mathcal{G}|$ is bounded below away from zero on $(\R + i [-\tilde{\lambda}, \tilde{\lambda}]) \times \R.$
 We can therefore define $\mathcal{G}^{-1/2}$ as a uniformly continuous function  on $(\R + i [- \tilde{\lambda}, \tilde{\lambda}]) \times \R;$ $\mathcal{G}(z, \xi)^{-1/2}$ is analytic in $z\in \R + i (- \tilde{\lambda}, \tilde{\lambda}), \xi \in \R.$  We can therefore extend \eqref{tpsi} and  define 
 \[ \Psi (z, \xi) = \frac{1}{\sqrt{\beta}}\mathcal{G}^{-1/2}(z, \xi) G(z, \xi) \  \  \  (z\in \R + i (-\tilde{\lambda}, \tilde{\lambda}),  \xi \in \R). \]
 \item[3.]  We use the above  extension and its property 
 \begin{eqnarray*}\label{fpo}
\begin{cases}
\Psi (z+1, \xi)=\Psi (z, \xi),\\
\Psi(z, \xi +1)= e^{2\pi i z} \Psi (z, \xi)
\end{cases}
\end{eqnarray*}
to prove exponential decay of $\hat{\phi}.$ To this end, by \eqref{ifz} and \eqref{mi} and Cauchy formula, we have
\begin{eqnarray*}
|\hat{\phi}(\xi)|  & = &  \left|\sqrt{\beta} \int_0^1\Psi (x, \beta \xi) dx\right|\\
& = & \sqrt{\beta} \left| \int_0^{\Lambda} \Psi (i\tau, \beta \xi) d\xi +\int_0^{1} \Psi (x+i\Lambda, \beta \xi) d\xi  + \int_{\Lambda}^{0} \Psi (1+i\tau, \beta \xi) d\xi \right|\\
& = & \sqrt{\beta} \left|\int_0^{1} \Psi (x+i\Lambda, \beta \xi) d\xi \right| \lesssim  e^{-\pi \Lambda \xi},  
\end{eqnarray*}
for  $\xi \geq 0$ and  some $0< \Lambda< \tilde{\lambda}.$    For $\xi \leq 0$ we  may use the similar argument, but we deform the integration path by going into the $\text{Im} z<0$ the half plane. 
 \end{enumerate}
\noindent \textbf{Step III}:  The $\phi$  has  an exponential decay. To achieve  this, we use the connection (Lemma  \ref{rzf}) between the Zak transforms of a function and of its Fourier transform and  similar procedure as in  the previous step. For the clarity, we briefly  highlight substeps:
\begin{enumerate}
\item[1.] $G$ can be extended to a uniformly continuous function  on $\R \times (\R+ i (\mu/4\pi, \infty))$, and that, for every $x \in \R,$  $G(x, \xi + i \sigma)$  is analytic in  $\xi +i\sigma \in \R +i (\mu/4\pi, \infty).$ 
\item [2.] We  define, for $x \in \R,$  $w= \xi + i \sigma \in \R + i (\mu/4\pi, \infty),$
\[ \Gamma(x, w) = \sum_{r=0}^{\beta^{-1}-1} G(x, w- \beta r) G(-x, w-\beta r). \]
Again $\Gamma(x, w)$ is analytic, and there exists  $\tilde{\mu}>0$  so that $|\Gamma|$ is bounded below away from zero on $\R \times (\R + i [-\tilde{\mu}, \tilde{\mu}]).$ It follows that $\Psi$ has an extension  to $\R \times (\R + i [-\tilde{\mu}, \tilde{\mu}]),$
\[ \Psi(x, \xi+i \sigma)=\frac{1}{\sqrt{\beta}} G(x, \xi +i \sigma)  \Gamma(x, \xi+ i\sigma)^{-1/2},\]
which is analytic in $\xi +i \sigma$ for every fixed $x,$ and which satisfies
\begin{eqnarray*}
\begin{cases} \Psi (x, w+1)= e^{2\pi i x} \Psi (x, w),\\
\Psi (x+1, w)=  \Psi (x, w).
\end{cases}
\end{eqnarray*}
\item[3.]   
By Lemma \ref{ifz} and \eqref{rzf}, we have 
\begin{eqnarray*}
 \phi(\xi) & = &  \sqrt{\beta}\int_0^1 Z_{\beta} \phi (y, \beta \xi) dy\\
 & = & \beta \sqrt{ \beta}   \sum_{j=0}^{\beta^{-2}-1} \int_0^1 e^{2 \beta \pi i \xi (j+y)} Z_{\beta} \hat{\phi} \left(-\beta^{-1}\xi, \frac{y+j}{\beta^{-2}}\right) dy\\
 & = & \beta \sqrt{ \beta}   \sum_{j=0}^{\beta^{-2}-1} \int_0^1 e^{2\beta \pi i \xi (j+y)} \Psi \left(-\beta^{-1}\xi, \frac{y+j}{\beta^{-2}}\right) dy.
\end{eqnarray*}
Now similarly to the the last part of Step II, we may obtain the the desired estimate.
\end{enumerate}
\noindent \textbf{Step IV}:  In view of \eqref{tpsi}, notice that 
\begin{eqnarray}\label{ku}
  \sum_{r=0}^{\beta^{-1}-1} \left| Z_{\beta} \hat{\phi}\left(x,   \xi- \beta r \right)\right|^2 & = & \sum_{r=0}^{\beta^{-1}-1} \left| \Psi \left(x,   \xi- \beta r \right)\right|^2  \nonumber \\
  & = &  \frac{1}{\beta} \sum_{r=0}^{\beta^{-1}-1}  \frac{\left|G(x, \xi-\beta r)\right|^2}{\sum_{\ell=0}^{\beta^{-1}-1}  \left| G \left(x, \xi -\beta \ell- \beta r \right)\right|^2 } \nonumber \\
  &=&  \frac{1}{\beta} \sum_{r=0}^{\beta^{-1}-1}  \frac{\left|G(x, \xi-\beta r)\right|^2}{\sum_{\ell=r}^{\beta^{-1}-1+r}  \left| G \left(x, \xi -\beta \ell \right)\right|^2 }.
\end{eqnarray}
By Lemma  \ref{qp}, for $r\geq 1,$ we notice
\begin{eqnarray*}
\sum_{\ell=r}^{\beta^{-1}-1+r}  \left| G \left(x, \xi -\beta \ell \right)\right|^2 & = &  \sum_{\ell=r}^{\beta^{-1}-1}  \left| G \left(x, \xi -\beta \ell \right)\right|^2 \\
&& + \left| G \left(x, \xi -\beta (\beta^{-1}-1 +1) \right)\right|^2 + \cdots  \left| G \left(x, \xi -\beta (\beta^{-1}-1 +r) \right)\right|^2\\
& = &   \sum_{\ell=r}^{\beta^{-1}-1}  \left| G \left(x, \xi -\beta \ell \right)\right|^2 \\
&& + \left| G \left(x, \xi -1 \right)\right|^2 + \cdots  \left| G \left(x, \xi - \beta (r-1) -1 \right)\right|^2\\
& = & \sum_{\ell=0}^{\beta^{-1}-1}  \left| G \left(x, \xi -\beta \ell \right)\right|^2.
\end{eqnarray*}
This together with \eqref{ku}, we have 
\[\sum_{r=0}^{\beta^{-1}-1} \left| Z_{\beta} \hat{\phi}\left(x,   \xi- \beta r \right)\right|^2 = \frac{1}{\beta}.\]
This together with preceding steps completes the proof.
 \end{proof}
 
 We are now ready to prove  Theorems \ref{hp}.
 
\begin{proof}[Proof of Theorem \ref{hp}]
Combining Theorems  \ref{wg} and \ref{ndua}, Remark \ref{FR}\eqref{FR2} and Proposition \ref{ciz},  the proof follows.

For the last part, we proceed as in the proof of the last part of Theorem~\ref{wg}. 
\end{proof}

\begin{proof}[Proof of Theorem~\ref{thm-main2}] 
Since Zak transform $Z_{\beta}:L^2(\mathbb R)\to \mathcal{Z}$ is surjective, to prove Theorem \ref{thm-main2}, it suffices to prove that there does not exist  any  generator $\phi,$ defined by \eqref{mi} (with any seed function $g$), which can convert the Wilson system \eqref{dks} into an ONB for  $L^2(\R)$ unless $\beta^{-1}=2.$   

We  shall prove this by contradiction.  If possible, suppose that there  exist   generator $\phi,$ defined by \eqref{mi} (for some seed function $g$), which can convert the Wilson system \eqref{dks} into an ONB for  $L^2(\R)$ and $\beta^{-1}\neq 2.$ Then by the last part of Theorem \ref{wg} (or Theorem~\ref{hp}), we must have  $\|\phi\|_{L^2}=\|\hat{\phi}\|_{L^2}= 1/\sqrt{2\beta}.$ On the other hand, using  \eqref{uzt}, we obtain
\begin{eqnarray*}
\|\hat{\phi}\|_{L^2(\R)}^2= \|Z_{\beta}^{-1}\Psi\|_{L^2(\R)}^2 & = & \| \Psi \|_{L^2([0, 1)^2)}^2 \\
& = &  \frac{1}{\beta}\int_0^1 \int_0^1  \frac{|Z_{\beta}g(x,\xi)|^2}{ \sum_{r=0}^{\beta^{-1}-1}  \left| Z_{\beta}g \left(x,  \xi - \beta r \right)\right|^2 } dx d\xi \\
& = & \beta^{-1} \|h\|_{L^2(\mathbb T)}^2,
\end{eqnarray*}
where 
$h(x, \xi) :=  \frac{G(x, \beta \xi)}{ \left(\sum_{r=0}^{\beta^{-1}-1}  \left| G \left(x, \beta \xi - \beta r \right)\right|^2 \right)^{1/2}},  (x, \xi \in \mathbb T). $
  Consider translation operator in the second variable $T_{\ell}:L^2(\mathbb T^2) \to L^2(\mathbb T^2): h(x,y)\mapsto h(x, \xi-\beta \ell), $  and we have  $\|T_{\ell} h\|_{L^2}= \|h\|_{L^2}$ for $\ell \in \mathbb N.$
This together with \eqref{qp}, we obtain
  \begin{eqnarray*}
 \beta^{-1} \|h\|_{L^2(\mathbb T^2)}^2 & = &   \|h\|^2_{L^2(\mathbb T^2)} + \sum_{\ell=1}^{\beta^{-1}-1} \|T_{\ell}h\|^2_{L^2(\mathbb T)}\\
 & = &  \int_0^1 \int_0^1  \frac{  \sum_{\ell =0}^{\beta^{-1}-1} |Z_{\beta}g(x,\xi-\beta \ell )|^2}{ \sum_{r=0}^{\beta^{-1}-1}  \left| Z_{\beta}g \left(x,  \xi - \beta r \right)\right|^2 } dx d\xi=1.
  \end{eqnarray*}
  Thus, we have  $\|\hat{\phi}\|_{L^2(\R)}^2=1$ a contradiction to  the hypothesis $\beta^{-1}\neq 2.$
\end{proof}

\section*{Acknowledgments.}  D. G. B.\  is grateful to Professor  Kasso Okoudjou for hosting and  arranging research facilities at the University of Maryland.  D. G. B. is  thankful to  SERB Indo-US Postdoctoral Fellowship (2017/142-Divyang G Bhimani) for the financial support. D.G.B.   would like to  express  many thanks  to Professor Pascal Aucher for sending  his paper \cite{auscher1994remarks}.  D.G.B. is also thankful to DST-INSPIRE and TIFR CAM for the  academic leave.  K. A. O.\ was partially supported by a grant from the Simons Foundation $\# 319197$,  the U. S.\ Army Research Office  grant  W911NF1610008,  the National Science Foundation grant DMS 1814253, and an MLK  visiting professorship.

\bibliographystyle{amsplain}
\bibliography{dkbibfile}

\end{document}